\documentclass[12pt,leqno]{amsart}

\usepackage{amssymb, amsmath, amsthm, esint}
\usepackage{bbm,dsfont}
\usepackage{tikz}
\usepackage{hyperref}

\title[variation-norm Stein--Wainger]{S\MakeLowercase{harp variation-norm estimates for oscillatory integrals related to} C\MakeLowercase{arleson's theorem}}

\author[S. Guo]{Shaoming Guo}
\address{Shaoming Guo: Department of Mathematics, Indiana University Bloomington, 107 S Indiana Ave, Bloomington, IN 47405, USA}
\address{Current address: 
Department of Mathematics, The Chinese University of Hong Kong, Ma Liu Shui, Shatin, Hong Kong}
\email{shaomingguo2018@gmail.com}

\author[J. Roos]{Joris Roos}
\address{Joris Roos: Department of Mathematics, University of Wisconsin-Madison, 480 Lincoln Dr, Madison, WI-53706, USA}
\email{jroos@math.wisc.edu}

\author[P.-L. Yung]{Po-Lam Yung}
\address{Po-Lam Yung: Department of Mathematics, The Chinese University of Hong Kong, Ma Liu Shui, Shatin, Hong Kong}
\email{plyung@math.cuhk.edu.hk}
\date{\today}


\def\R{\mathbb{R}}
\def\N{\mathbb{N}}
\def\C{\mathbb{C}}
\def\Z{\mathbb{Z}}
\def\beq{\begin{equation}}
\def\endeq{\end{equation}}
\def\lesim{\lesssim}
\DeclareMathOperator{\Mod}{Mod}

\theoremstyle{plain}
\newtheorem{thm}{Theorem}[section]
\newtheorem{prop}[thm]{Proposition}
\newtheorem{lem}[thm]{Lemma}
\newtheorem{cor}[thm]{Corollary}

\numberwithin{equation}{section}
\begin{document}

\begin{abstract}
We prove variation-norm estimates for certain oscillatory integrals related to Carleson's theorem. Bounds for the corresponding maximal operators were first proven by Stein and Wainger. Our estimates are sharp in the range of exponents, up to endpoints. Such variation-norm estimates have applications to discrete analogues and ergodic theory. The  proof relies on square function estimates for Schr\"odinger-like equations due to Lee, Rogers and Seeger. In dimension one, our proof additionally relies on a local smoothing estimate. Though the known endpoint local smoothing estimate by Rogers and Seeger is more than sufficient for our purpose, we also give a proof of certain local smoothing estimates using Bourgain--Guth iteration and the Bourgain--Demeter $\ell^2$ decoupling theorem. This may be of independent interest, because it improves the previously known range of exponents for spatial dimensions $n\ge 4$.
\end{abstract}
\maketitle

\section{Introduction}

Let $n \geq 1$ and $\alpha > 1$ be fixed. Given a  Calder\'{o}n--Zygmund kernel $K: \R^n\to \R$ we define a modulated singular integral by
\begin{equation} \label{eq:Hu_def}
\mathcal{H}^{(u)} f(x) := \int_{\R^n} f(x-t) e^{i u |t|^{\alpha}} K(t)dt, \quad u \in \R.
\end{equation}
The maximal operator
\begin{equation}\label{max-operator}
\sup_{u \in \R} |\mathcal{H}^{(u)} f|
\end{equation}
was introduced in Stein and Wainger \cite{MR1879821}, as a generalization of the Carleson operator studied in Carleson \cite{MR0199631}, Fefferman \cite{MR0340926}, Lacey and Thiele \cite{MR1783613} and so on. In this paper, we study variation-norm estimates for the family $\{\mathcal{H}^{(u)} f\}_{u \in \R}$. Apart from the intrinsic interest in such bounds, another strong motivation is given by the connection to certain discrete analogues of \eqref{max-operator} that are the subject of recent works by Krause and Lacey \cite{MR3658135}, \cite{Kra18} (see Section \ref{sec:discranalogue} below).

If $\mathcal{J}$ is a subset of $\R$ and $\{a_u \colon u \in \mathcal{J}\}$ is a family of complex numbers indexed by $\mathcal{J}$, then for any $1 \leq r < \infty$ the $r$-variational norm of $\{a_u\}_{u \in \mathcal{J}}$ is defined to be
\[
V^r \{a_u \colon u \in \mathcal{J}\} := \sup_{J \in \N} \, \sup_{ \substack{u_0, u_1, \dots, u_J \in \mathcal{J} \\ u_0 < u_1 < \dots < u_J}} \left( \sum_{j=1}^J |a_{u_j} - a_{u_{j-1}}|^r \right)^{1/r}.
\]
Closely related to it is the jump function of the family $\{a_u\}_{u \in \mathcal{J}}$: For $\lambda > 0$, the $\lambda$-jump function of $\{a_u\}_{u \in \mathcal{J}}$, namely $N_{\lambda}\{a_u \colon u \in \mathcal{J}\}$, is defined to be the supremum of all positive integers $N$ for which there exists a strictly increasing sequence $s_1 < t_1 < s_2 < t_2 < \dots < s_N < t_N$, all of which are in $\mathcal{J}$, such that 
\[
|a_{t_j} - a_{s_j}| > \lambda
\]
for all $j=1,\dots,N$. For $r \in (1,\infty)$ and $p \in (1,\infty)$, we will study the $L^p$ mapping properties of the maps
\[
f \longmapsto V^r \{\mathcal{H}^{(u)} f\colon u \in \R\}
\]
and
\[
f\longmapsto \lambda [N_{\lambda} \{\mathcal{H}^{(u)} f\colon u \in \R\}]^{1/r}, \quad \lambda > 0.
\]
Henceforth $f$ will always be a Schwartz function on $\R^n$; the goal is to establish \emph{a priori} bounds for all such $f$.
If in dimension $n = 1$ we take $\alpha = 1$ and replace $|t|$ by $t$, then this corresponds to the variation-norm Carleson operator, which has been studied in \cite{MR2881301} and \cite{Gen16}. We refer the reader to \cite{MR1019960}, \cite{MR933985}, \cite{MR1788042}, \cite{MR1953540}, \cite{MR2434308} and the references therein for earlier results concerning jump function and variation-norm inequalities for other operators arising in harmonic analysis.\\

Let us assume that $K$ is a homogeneous Calder\'{o}n--Zygmund kernel, in the sense that
\[ \label{eq:CZ}
K(x) = \mathrm{p.v.}\frac{\Omega(x)}{|x|^n}
\]
for some function $\Omega$ that is smooth on $\R^n \setminus \{0\}$, homogeneous of degree 0. The assumption that $K$ is homogeneous is not strictly necessary. It is there to help simplify the presentation of the proof of the theorem. 
We also assume that $\int_{\mathbb{S}^{n-1}} \Omega(x) d\sigma(x) = 0$, where $\sigma$ denotes the surface measure on $\mathbb{S}^{n-1}$.

\begin{thm}\label{main}
Let $n \geq 1$, $\alpha \in (1,\infty)$ and define $\mathcal{H}^{(u)}$ as in (\ref{eq:Hu_def}). If $r \in (2,\infty)$, $p \in (1,\infty)$ and $r > \frac{p'}{n}$, then we have
\begin{equation}\label{170614e1.2}
\left\|V^r \{\mathcal{H}^{(u)} f: u\in \R\} \right\|_p \le C \|f\|_p.
\end{equation}
In addition, if $n \ge 2$ and $p \in \left( \frac{2n}{2n-1}, \infty \right)$, then
\[
\left\|\lambda \sqrt{N_{\lambda} \{\mathcal{H}^{(u)} f \colon u \in \mathbb{R}\}}\right\|_p \le C \|f\|_p.
\]
Here the constant $C$ is allowed to depend on $n, \alpha, p$ and $r$.
\end{thm}

Moreover, up to endpoints, we show that this is the best we can expect: 
\begin{thm}\label{170712thm1.2}
	The estimate \eqref{170614e1.2} fails if $r < \frac{p'}{n}$.
\end{thm} 

Thus, the range of exponents for which estimate  \eqref{170614e1.2} holds is given by the quadrilateral in Figure~\ref{fig:exp} below (up to endpoints).

\begin{figure}[ht]
	\label{fig:exp}
	\begin{tikzpicture}[scale=4]
	\draw (0,0) [->] -- (1.3,0) node [below] {$\frac1p$};
	\draw (0,0) [->] -- (0,0.7) node [left] {$\frac1r$};
	
	\node at (0,.5) [left] {$\frac12$};
	
	\node at (.75,0) [below] {$\frac{2n-1}{2n}$};
	\node at (1,0) [below] {$1$};
	\node at (0,0) [below] {$0$};
	
	\draw (0,.5) -- (.75, .5) -- (1,0);

	\draw [dashed] (.75,0) -- (.75,.5); 
	
	\end{tikzpicture}
	\caption{}
\end{figure}

It is a natural question what happens when $\alpha$ is less than 1. Our methods do not seem to be able to handle this case. But if $n = 1$, an easy adaptation of our methods allows us to obtain a positive result where the phase function $|t|^{\alpha}$ in (\ref{eq:Hu_def}) is replaced by $\text{sgn}(t) |t|^{\alpha}$. In particular, if $\alpha$ is an odd positive integer, we may replace $|t|^{\alpha}$   in (\ref{eq:Hu_def}) by $t^{\alpha}$ and still obtain a positive result.

The inequality \eqref{170614e1.2} can be understood as an extension of Stein and Wainger's well-known result from \cite{MR1879821} (also see \cite{MR3470421} for the case when $\alpha$ is not an integer):
\begin{equation}\label{170712e1.5}
\Big\|\sup_{u\in \R} |\mathcal{H}^{(u)} f| \Big\|_p \lesssim \|f\|_p, \text{ for every } p>1.
\end{equation}

\subsection{Connection with discrete analogues}\label{sec:discranalogue}
Further motivation stems from the study of a discrete analogue of the maximal operator \eqref{max-operator}. Fix an integer $d\ge 2$ and let $u\in\R$. Consider the following operator $\mathcal{H}^{(u)}_\Z$ acting on functions $f:\Z\to\C$,
\[ \mathcal{H}_\Z^{(u)} f(x) = \sum_{t\in\Z\setminus\{0\}} f(x-t) e^{iu t^d} \frac1{t},\quad x\in\Z. \]
This is a discrete analogue of our operator $\mathcal{H}^{(u)}$ for $n=1$ and $\alpha=d$. Bounding the associated maximal operator $f\mapsto \sup_{u\in\R} |\mathcal{H}_\Z^{(u)} f|$ on $\ell^p(\Z)$ is significantly more difficult than bounding Stein and Wainger's maximal operator and until recently, no such bounds were known. For the recent progress on this problem and further discussion of discrete analogues, we refer to Krause \cite{Kra18} and earlier work by Krause and Lacey \cite{MR3658135}.
A careful analysis of the multiplier of $\mathcal{H}_\Z^{(u)}$, which is much in the spirit of the Hardy-Littlewood circle method, reveals a natural splitting of the problem into a number-theoretic and an analytic component. 
In the case $p=2$, the core estimate for the analytic component is a variant of Bourgain's classical maximal multi-frequency lemma \cite[Lemma 4.1]{MR1019960}. The precise statements can be found in \cite[Section 3]{MR3658135} and \cite[Sections 5 and 10.2]{Kra18}; see, in particular, Theorem 3.5 of \cite{MR3658135}. 
Using a small refinement of our Theorem \ref{main} (see Theorem~\ref{thm:mainsharpCp} below), together with Bourgain's argument from \cite{MR1019960}, one can obtain an alternative simple proof of (a small extension of) Theorem~3.5 of \cite{MR3658135}; we include some details in an appendix below (see Section~\ref{sect:KL}). 

Discrete analogues are intimately related to ergodic theorems and this connection provides a further application of our variation-norm estimates. Krause made use of a variant of the estimate \eqref{170614e1.2} in his recent work on a pointwise ergodic theorem \cite[Theorem 1.2]{Kra18}.

\subsection{Outline of the proof}

We now briefly describe an outline of the proof of Theorem~\ref{main}. 
To control the left hand side of the estimate \eqref{170614e1.2}, we split the contribution into two parts: \emph{long variations} and \emph{short variations}. For each $j \in \Z$, define the short variation on the $u$-interval $[2^{j\alpha}, 2^{(j+1)\alpha}]$ by
\[
V^{r}_{ j}{\mathcal{H} f}(x)
:= V^{r} \{\mathcal{H}^{(u)} f (x) \colon u \in [2^{j\alpha}, 2^{(j+1)\alpha}]\}.
\]
Also define
\[
S_{r} (\mathcal{H}f)(x) :=\left(\sum_{j\in \Z} |V_j^{r} \mathcal{H}f (x)|^{r} \right)^{\frac{1}{r}}
\]
and
\[
N^{\mathrm{dyad}}_{\lambda}(\mathcal{H}f)(x) :=N_{\lambda}\{\mathcal{H}^{(2^{j \alpha})}f(x): j\in \Z\}.
\]
We will use the following lemma (see, for example, Jones, Seeger and Wright \cite{MR2434308}):
\begin{lem}
For $r\in[2,\infty)$ we have 
\[
\lambda [N_{\lambda} \{\mathcal{H}^{(u)} f \colon u > 0\}]^{1/r} \lesssim S_{r}(\mathcal{H}f)+\lambda [N_{\lambda/3}^{\mathrm{dyad}}(\mathcal{H}f)]^{1/r},
\]
uniformly in $\lambda > 0$.
\end{lem}
(Hereafter, $A \lesssim B$ means $A \leq C B$ for some absolute constant $C$.) 

By this lemma, and by Bourgain's argument \cite{MR1019960} of passing from jump norms to variation-norms (see also \cite[Section 2]{MR2434308}), to prove Theorem \ref{main} it suffices to prove the following two propositions. 
\begin{prop}\label{main1}
For every $p\in (1,\infty)$  and $r\in [2,\infty)$ we have 
\[
\|\lambda [N_{\lambda/3}^{\mathrm{dyad}}(\mathcal{H}f)]^{1/r}\|_p \lesssim \|f\|_p,
\]
uniformly in $\lambda > 0.$
\end{prop}

\begin{prop}\label{main2}
Let $n \ge 1$ and $p\in(1, \infty)$, $r\in (2,\infty)$ with $r>p'/n$. Then we have
\[
\|S_{r}(\mathcal{H} f)\|_p \lesssim \|f\|_p.
\]
If $n\ge 2$, then the inequality also holds for $r=2$.
\end{prop}

The proof of Proposition~\ref{main1} depends on a jump function inequality of Jones, Seeger and Wright \cite{MR2434308} that is based on a L\'epingle inequality for martingales. 

By interpolation with the inequality \eqref{170712e1.5} of Stein and Wainger \cite{MR1879821}, it suffices to consider the case $p\in(\frac{2n}{2n-1},\infty)$ to prove Proposition~\ref{main2}. The proof of Proposition~\ref{main2} then depends on a square function estimate for Schr\"odinger-like equations, which is due to Lee, Rogers and Seeger \cite{MR2946085}. In one dimension, we additionally need a local smoothing estimate for these equations. 
The following local smoothing result is more than sufficient for our needs: indeed we will only need the following estimate for $n = 1$ and some $p < \infty$. We are including the full theorem here only because it may be of independent interest.

\begin{thm}\label{180320thm1.6}
Let $\gamma>1$ be a real number and let $I$ be a compact time interval. For any dimension $n\ge 1$ and exponent $p < \infty$ satisfying 
\beq\label{p-exponent}
\begin{cases}
\hfill p>\frac{2(4n+7)}{4n+1}  & \hfill \text{ if } n\equiv -1 \mod 3\\
\hfill p>\frac{2n+3}{n}  & \hfill \text{ if } n\equiv 0 \mod 3\\
\hfill p>\frac{4(n+2)}{2n+1} & \hfill \text{ if } n\equiv 1 \mod 3
\end{cases}
\endeq
we have
\beq \label{eq:ls}
\left(\int_{\R^n\times I} \left| \int_{\R^n}e^{ix \cdot \xi}\hat{f}(\xi)e^{it|\xi|^{\gamma}} d\xi\right|^p dxdt\right)^{1/p} \lesim_{\epsilon} \|f\|_{W^{\beta+\epsilon,p}(\R^n)},
\endeq
whenever $\epsilon > 0$ and $$\frac{\beta}{\gamma}=n\left (\frac{1}{2}-\frac{1}{p}\right )-\frac{1}{p}.$$ Here we write $W^{s,p}(\R^n) = (I-\Delta)^{-s/2}L^p(\R^n)$ to denote the standard Bessel potential space.
\end{thm}

Let us take a moment to compare Theorem~\ref{180320thm1.6} with results in the existing literature. In \cite{Rog08}, Rogers considered the case $\gamma = 2$, namely a local smoothing estimate for the Schr\"{o}dinger propagator $e^{it\Delta}$. He proved that (\ref{eq:ls}) holds whenever $\gamma = 2$, $p \in (2+\frac{4}{n+1},\infty)$ and $\epsilon > 0$ (in the rest of this section $\beta$ will always be as specified in Theorem~\ref{180320thm1.6}). This was improved subsequently in Rogers and Seeger \cite{MR2629687}, who obtained the endpoint case $\epsilon = 0$ for all $\gamma > 1$: they established that (\ref{eq:ls}) holds with $\varepsilon = 0$ for all $p \in (2+\frac{4}{n+1},\infty)$ and all $\gamma > 1$. In particular, this implies Theorem \ref{180320thm1.6} for $n=1,2,3$. Theorem \ref{180320thm1.6} gives a larger range of $p$ in dimensions $n\ge 4$, albeit with an $\varepsilon$-loss in smoothness. We also note that in the case $\gamma = 2$ (i.e. for the Schr\"{o}dinger propagator), Lee, Rogers and Seeger \cite{LRS13} obtained an improvement of the aforementioned result of Rogers and Seeger \cite{MR2629687}; in particular, in Proposition 5.2 of \cite{LRS13}, they proved that if the dual Fourier restriction conjecture holds at exponent $q_0$, in the sense that $$\|Ef\|_{L^{q_0}(\R^{n+1})} \lesssim \|f\|_{L^{q_0}([0,1]^n)}$$ for some exponent $\gamma_0 < \frac{2(n+3)}{n+1}$, where $E$ is the Fourier extension operator for the paraboloid in $\R^{n+1}$ given by
\begin{equation} \label{eq:extconj}
Ef(x,t) = \int_{[0,1]^n} f(\xi) e^{i (x \cdot \xi + t |\xi|^2)} d\xi, \quad (x,t) \in \R^n \times \R,
\end{equation}
then (\ref{eq:ls}) holds for $\gamma = 2$ with $\epsilon = 0$ whenever $p \in (q_*,\infty)$, where $q_*$ is defined by
$$
q_* := \frac{2(n+3)}{n+1} \left(1- \gamma(n,q_0) \right), \quad \text{with} \quad \gamma(n,q_0) := \frac{ \frac{1}{q_0} - \frac{n+1}{2(n+3)} }{ n \left( \frac{n+1}{2} - \frac{n+2}{q_0} \right) }.
$$
A direct computation shows that
$$
q_* = 2 + \frac{4}{n} - \frac{2}{n^2 - \frac{4-q_0}{q_0-2} n}.
$$
As a result, even if one can establish (\ref{eq:extconj}) in all dimensions $n$ with $q_0 = q_0(n)$ that decays like $q_0(n) = 2 + \frac{2+\lambda}{n} + O(\frac{1}{n^2})$ for some $\lambda > 0$ (the Fourier restriction conjecture shows that the best one can hope for is $\lambda = 0$), using the above result of Lee, Rogers and Seeger, one can only establish the local smoothing estimate (\ref{eq:ls}) for $p \in (q_*(n),\infty)$ where $$q_*(n) = 2 + \frac{4}{n} + O\left(\frac{1}{n^2}\right).$$ On the contrary, if $p_*(n)$ is the Bourgain-Guth exponent given by the right hand sides of (\ref{p-exponent}), we see that 
$$
p_*(n) = 2 + \frac{3}{n} + O\left(\frac{1}{n^2} \right),
$$
so our range of the exponent $p$ is larger than that of Lee, Rogers and Seeger in high dimensions $n$, even for the Schr\"{o}dinger equation case.

Contrary to the work of Rogers and Seeger \cite{MR2629687}, which relied on bilinear restriction estimates, our proof of Theorem \ref{180320thm1.6} relies on the Bourgain--Guth argument \cite{MR2860188} (see also the presentation in \cite{MR3592159}), and the Bourgain--Demeter decoupling inequality \cite{MR3374964}; see Wolff \cite{Wol00} and {\L}aba and Wolff \cite{LW02} for some earlier foundational work on decoupling inequalities, and their applications to local smoothing estimates. The multilinear estimates developed by Guth \cite{Guth16} might be useful in establishing \eqref{eq:ls} for a larger range of exponents, but we did not pursue this here. \\

\noindent{\bf Organization of the paper.} In Section \ref{sect:prelim} we state two preliminary results, namely a consequence of the classical L\'{e}pingle inequality, and a consequence of the Plancherel--P\'{o}lya inequality. In Section \ref{section:long-jumps} we control long jumps: that is, we will prove Proposition \ref{main1}. The treatment for short jumps (that is, the proof of Proposition \ref{main2}), will be split into two parts. In Section \ref{section:short-jump-large} we prove Proposition \ref{main2} in two special cases: $n \geq 2$, $p > \frac{2(n+2)}{n}$, and $n = 1$, $p > 2$. These are the main cases to be considered. In Section \ref{section:short-jump-small} we indicate the modifications necessary to prove the remaining case of Proposition \ref{main2}: namely, $n \geq 2$ and $\frac{2n}{2n-1}<p\leq \frac{2(n+2)}{n}$. The proof of Theorem \ref{170712thm1.2} is in Section \ref{section:counter}. In Section \ref{180320section7} we provide the proof of a vector-valued generalization of a multiplier theorem of Seeger \cite{MR955772}, that we used in the proof of the short jump estimates in Section \ref{section:short-jump-large}. In Section \ref{180320section8} we prove the local smoothing estimates in Theorem \ref{180320thm1.6}. In Section \ref{sect:KL} we refine our Theorem~\ref{main} by obtaining a good bound on the growth of the constant $C$ in \eqref{170614e1.2} as $p=r \to 2^+$ (see Theorem~\ref{thm:mainsharpCp}), and use it to provide an alternative simple proof of a maximal multi-frequency estimate of Krause and Lacey (that is, the Theorem 3.5 in \cite{MR3658135}). \\

\noindent{\bf Acknowledgements.} The authors would like to express their gratitude to Andreas Seeger for many valuable discussions. We also thank Chun-Kit Lai for helpful conversations, and the referees for numerous very constructive comments. Part of this material is based upon work supported by the National Science Foundation under Grant No. DMS-1440140 while the authors were in residence at the Mathematical Sciences Research Institute in Berkeley, California, during the Spring semester of 2017. Guo was also partially supported by a direct grant for research from the Chinese University of Hong Kong (4053295). Roos was also partially supported by the German Academic Scholarship Foundation. Yung was also partially supported by a General Research Fund CUHK14303817 from the Hong Kong Research Grant Council, and direct grants for research from the Chinese University of Hong Kong (3132713, 4053295 and 4441881). 

\section{Prerequisites} \label{sect:prelim}

\subsection{A jump function inequality of Jones, Seeger and Wright}

We recall a jump function inequality for convolutions with dyadic dilations of a fixed measure from \cite[Theorem 1.1]{MR2434308}. It is a consequence of the more classical L\'epingle inequality for martingales. 

\begin{prop}[Jones, Seeger and Wright \cite{MR2434308}] \label{prop:Lepingle}
Let $\sigma$ be a compactly supported finite non-negative Borel measure on $\R^n$ whose Fourier transform satisfies
\[|\widehat{\sigma}(\xi)| \leq C |\xi|^{-a}\] for some $a > 0$. For $k \in \Z$, define $\sigma_k$ by 
\[
\int_{\R^n} f(x)d\sigma_k(x)=\int_{\R^n}f(2^{-k} x)d\sigma(x).
\] 
Then 
\[
\left\| \lambda \sqrt{N_{\lambda}\{ f*\sigma_k \colon k \in \Z \}} \right\|_{L^p(\R^n)} \leq C_p \|f\|_{L^p(\R^n)}
\]
for all $1 < p < \infty$, uniformly in $\lambda > 0$.
\end{prop}

We will apply this proposition as follows. Let $S$ be a non-negative smooth function with compact support in $[-1,1]^n$ and $\int_{\R^n} S(x) dx = 1$. 
For $k \in \Z$ and any Schwartz function $f$ on $\R^n$, let 
\[S_k f(x) = f*S_k(x),\]
where $S_k(x)= 2^{kn} S(2^k x)$.
If $\sigma$ is the measure on $\R^n$ given by 
\[
\int_{\R^n} f(x) d\sigma(x) = \int_{\R^n} f(x) S(x) dx,
\]
then $\sigma_k(x)$ coincides with $S_k(x) dx$, and hence $f * \sigma_k = S_k f$ for all $k \in \Z$. Proposition~\ref{prop:Lepingle} then gives
\begin{equation} \label{eq:Sk_jump}
\left\| \lambda \sqrt{N_{\lambda}\{ S_k f \colon k \in \Z \}} \right\|_{L^p(\R^n)} \leq C_p \|f\|_{L^p(\R^n)} 
\end{equation}
for all $1 < p < \infty$, uniformly in $\lambda > 0$.  Note that $\widehat{S}(0) = 1$ and $\widehat{S}(\xi)$ decreases rapidly to zero as $|\xi| \to \infty$. So later it helps to think of $\widehat{S}(\xi)$ as localized to $|\xi| \lesssim 1$, and interpret $S_k f$ as a localization of $f$ to frequency $\lesssim 2^k$.\\

Next, let $\{c_{\ell}\}_{\ell=0}^{\infty}$ be a complex sequence with $|c_{\ell}| = O(2^{-\alpha \ell})$ for some $\alpha > 0$. Let $\tilde{S}_k$ be the operator defined by
\begin{equation} \label{eq:Tk_def}
\tilde{S}_k f := \sum_{\ell = 0}^{\infty} c_{\ell} S_{k-\ell} f.
\end{equation}
We will use \eqref{eq:Sk_jump}
 to prove that 
\begin{equation} \label{eq:clSk_jump}
\left\| \lambda \sqrt{N_{\lambda}\{ \tilde{S}_k f \colon k \in \Z \}} \right\|_{L^p(\R^n)} \leq C_p \|f\|_{L^p(\R^n)}
\end{equation}
for all $1 < p < \infty$, uniformly in $\lambda > 0$. Recall the definition of the jump norm $N_{\lambda}\{ \tilde{S}_k f(x) \colon k \in \Z \}$: it is the supremum of all positive integers $N$ for which there exists a strictly increasing sequence $s_1 < t_1 < s_2 < t_2 < \dots < s_N < t_N$, all of which are in $\Z$, such that 
\begin{equation} \label{eq:clSk_jump2}
|\tilde{S}_{t_j}f(x) - \tilde{S}_{s_j}f(x)| > \lambda
\end{equation}
for all $j=1,\dots,N$. But if $s_1 < t_1 < s_2 < t_2 < \dots < s_N < t_N$ is as such, then for all $j=1,\dots,N$ we have
\[
|S_{t_j-\ell}f(x) - S_{s_j-\ell}f(x)| \gtrsim 2^{\frac{\alpha}{2} \ell}\lambda
\]
for at least one $\ell \ge 0$.
Hence, 
\[
N_{\lambda}\{ \tilde{S}_k f(x) \colon k \in \Z \} \lesssim \sum_{\ell = 0}^{\infty} N_{2^{\frac{\alpha}{2} \ell}\lambda}\{ S_k f(x) \colon k \in \Z \},
\]
which implies that
\[
\sqrt{N_{\lambda}\{ \tilde{S}_k f \colon k \in \Z \}}\lesssim \sum_{\ell = 0}^{\infty} \sqrt{N_{2^{\frac{\alpha}{2} \ell}\lambda}\{ S_k f \colon k \in \Z \}}.
\]
This further implies that
\[
\begin{split}
& \left\| \lambda \sqrt{N_{\lambda}\{ \tilde{S}_k f \colon k \in \Z \}} \right\|_{L^p(\R^n)} \lesssim \sum_{\ell = 0}^{\infty} 2^{-\frac{\alpha}{2} \ell} \left\|2^{\frac{\alpha}{2} \ell}\lambda \sqrt{N_{2^{\frac{\alpha}{2} \ell}\lambda}\{ S_k f \colon k \in \Z \}}\right\|_{L^p(\R^n)} \lesssim \|f\|_p.
\end{split}
\]
This finishes the proof of the estimate \eqref{eq:clSk_jump}. \\

\subsection{An inequality of Plancherel and P\'{o}lya}

Next, let $F(u)$ be an $L^2$ function on $\R$ whose Fourier transform $\widehat{F}(\xi)$ is supported on the set $|\xi| \leq 1$. Such an $F$ is sometimes said to be in a Paley-Wiener space. An inequality of Plancherel and P\'{o}lya \cite{PP37}, \cite{PP38} says that for any such $F$ and any $r \in [1,\infty)$, we have
\begin{equation} \label{eq:PP}
\sum_{j \in \Z} |F(j)|^r \leq C_r \int_{\R} |F(u)|^r du
\end{equation}
where $C_r$ is a constant independent of $F$. This holds because if $\widehat{F}$ is supported on $|\xi|\le 1$, then, by the uncertainty principle, $F$ is essentially constant on every interval of length one (see also  \cite{MR1836633} for an alternative proof based on complex analysis).

From (\ref{eq:PP}) we can deduce the following variation-norm estimate (see also page 6729 of \cite{MR2434308}):

\begin{prop}\label{PPinequality}
Let $F(u)$ be a function on $\R$ whose Fourier transform $\hat{F}(\xi)$ is supported on the set $\{|\xi| \leq \lambda\}$. Then for every $1 \leq q \leq r < \infty$, we have
\begin{equation} \label{eq:Variation}
V^r \{F(u) \colon u \in \R\} \leq A_{q,r} \lambda^{1/q} \|F\|_{L^q},
\end{equation}
with a constant $A_{q,r}$ depending only on $q$ and $r$.
\end{prop}

\begin{proof}By rescaling we may assume that $\lambda = 1$. Now let $k \in \N$  and $u_1 < \dots < u_k$ be a strictly increasing sequence in $\R$. We let $\kappa(0)=1$, $n_1 = \lfloor u_{\kappa(0)} \rfloor$ and let $\kappa(1)$ be the largest integer in $\{1,\dots,k\}$ such that $u_{\kappa(1)} < n_1 + 1$. If $\kappa(1) < k$, we let $n_2 = \lfloor u_{\kappa(1)+1} \rfloor$ and let $\kappa(2)$ be the largest integer  in $\{1,\dots,k\}$ such that $u_{\kappa(2)} < n_2 + 1$. Clearly this process will terminate in finitely many, say $m$, steps. In this way we collect the points $u_1, \dots, u_k$ into intervals $[n_1, n_1+1]$, $[n_2,n_2+1]$, $\dots$, $[n_m,n_m+1]$ of length at most~1.  Now for $s = 1,\dots,m-1$, we have, by the triangle inequality, that
\begin{align*}
&\, |F(u_{\kappa(s)})-F(u_{\kappa(s)+1})| ^r \\
\lesssim & \,   |F(u_{\kappa(s)})-F(n_s+1)|^r + |F(n_s+1)|^r + |F(n_{s+1})|^r + |F(n_{s+1})-F(u_{\kappa(s)+1})|^r.
\end{align*}
This shows that 
\begin{align*}
& \sum_{i=1}^{k-1} |F(u_i)-F(u_{i+1})|^r \lesssim \sum_{s = 1}^m \left( |F(n_s)|^r + |F(n_s+1)|^r \right) \\
&\quad + \sum_{s = 1}^m \left( |F(n_s)-F(u_{\kappa(s-1)})|^r + \sum_{\kappa(s-1) \leq i < \kappa(s)} |F(u_i)-F(u_{i+1})|^r + |F(u_{\kappa(s)})-F(n_s+1)|^r \right).
\end{align*}
(Indeed, for $s = 1$, we do not need the terms $|F(n_s)|^r$ and $|F(n_s)-F(u_{\kappa(s-1)})|^r$ on the right hand side; similarly for $s = m$, we do not need the terms $|F(n_s+1)|^r$ and $|F(u_{\kappa(s)})-F(n_s+1)|^r$. But there is no harm putting them in, which makes the expression on the right hand side more symmetric.) By the mean-value theorem, for $s = 1, \dots, m$, we have
\begin{align*}
& |F(n_s)-F(u_{\kappa(s-1)})|^r + \sum_{\kappa(s-1) \leq i < \kappa(s)} |F(u_i)-F(u_{i+1})|^r + |F(u_{\kappa(s)})-F(n_s+1)|^r \\
 \leq & \|F'\|_{L^{\infty}}^r \left(  |n_s-u_{\kappa(s-1)}|^r + \sum_{\kappa(s-1) \leq i < \kappa(s)} |u_i-u_{i+1}|^r + |u_{\kappa(s)}-(n_s+1)|^r \right),
\end{align*}
and the bracket in the last line is $\leq 1$ since we have the elementary inequality \[t_1^r + \dots + t_{\sigma}^r \leq (t_1+\dots+t_{\sigma})^r\] whenever $t_1, \dots, t_{\sigma} \geq 0$ and $1 \leq r < \infty$. Now since $\widehat{F}$ is supported on $|\xi| \leq 1$, Bernstein's inequality implies that \[\|F'\|_{L^{\infty}} \lesssim_r \|F\|_{L^r}\]  whenever $1 \leq r < \infty$. Altogether, we see that
\begin{align*}
\sum_{i=1}^{k-1} |F(u_i)-F(u_{i+1})|^r 
&\lesssim_r \|F\|_{L^r}^r + \sum_{s = 1}^m \left( |F(n_s)|^r + |F(n_s+1)|^r \right)  \\
&\lesssim_r \|F\|_{L^r}^r + \sum_{j \in \Z} |F(j)|^r \lesssim_r \|F\|_{L^r}^r
\end{align*}
whenever $1 \leq r < \infty$, the last inequality following from (\ref{eq:PP}). Since $\widehat{F}$ is supported on $\{|\xi| \leq 1\}$ and $1 \leq q \leq r$, Bernstein's inequality again implies that $\|F\|_{L^r} \lesssim_{q,r} \|F\|_{L^q}$. This completes the proof of (\ref{eq:Variation}). 
\end{proof}

\section{Long jump estimates}\label{section:long-jumps}

Our goal in this section is to prove Proposition \ref{main1}. Indeed, we will prove something slightly stronger, including the case $0<\alpha<1$.

\begin{prop}
Fix $\alpha > 0$, $\alpha \ne 1$. For $1 < p < \infty$, we have
\begin{equation} \label{eq:Hk_jump}
\| \lambda \sqrt{N_{\lambda} \{\mathcal{H}^{(2^{k\alpha})} f \colon k \in \Z\} } \|_{L^p(\R^n)} \lesssim \|f\|_{L^p(\R^n)}
\end{equation}
uniformly in $\lambda > 0$. Here $\mathcal{H}^{(2^{k\alpha})}$ is defined as in (\ref{eq:Hu_def}).
\end{prop}

First we decompose $\mathcal{H}^{(2^{k\alpha})}$ into 
\[
\begin{split}
\mathcal{H}^{(2^{k\alpha})} f(x)&= \int_{|t| \leq 2^{-k}} f(x-t) e^{i 2^{k\alpha} |t|^{\alpha}} K(t)dt+ \int_{|t| > 2^{-k}}  f(x-t) e^{i 2^{k\alpha} |t|^{\alpha}} K(t)dt\\
& =: \mathcal{H}_{k, -\infty} f(x)+ \mathcal{H}_{k, \infty} f(x).
\end{split}
\]
In the term $\mathcal{H}_{k, -\infty}f$, we are integrating over small $t$, and the exponential $e^{i 2^{k\alpha} |t|^{\alpha}}$ is approximately $1$. This motivates us to further decompose $\mathcal{H}_{k,-\infty} f$, as
\begin{equation} \label{eq:Hk0_decomp}
\begin{split}
\mathcal{H}_{k,-\infty}f(x) &=\int_{|t| \leq 2^{-k}} f(x-t) K(t)dt+\int_{|t| \leq 2^{-k}} f(x-t) (e^{i 2^{k\alpha} |t|^{\alpha}}-1) K(t)dt\\
& =:\widetilde{\mathcal{H}}_{k,0} f(x) + \mathcal{H}_{k,0} f(x).
\end{split}
\end{equation}
For the other term, we decompose 
\[ 
\mathcal{H}_{k,\infty} f(x)=
\sum_{\ell = 1}^{\infty} \mathcal{H}_{k,\ell} f(x) :=\sum_{\ell = 1}^{\infty} \int_{2^{-k+\ell-1} < |t| \leq 2^{-k+\ell}} f(x-t) e^{i 2^{k\alpha} |t|^{\alpha}} K(t)dt.
\]
The former term in \eqref{eq:Hk0_decomp} is a truncated singular integration. We have: 
\begin{lem}[Campbell, Jones, Reinhold, Wierdl \cite{MR1953540}, Theorem A]
\[ \left\| \lambda \sqrt{ N_{\lambda} \{\widetilde{\mathcal{H}}_{k,0} f \colon k \in \Z\} } \right\|_{L^p(\R^n)} \lesssim \|f\|_{L^p(\R^n)}
\] for all $1 < p < \infty$.
\end{lem}
Hence it remains to estimate the jump norms of $\mathcal{H}_{k,0} f(x) + \sum_{\ell=1}^{\infty} \mathcal{H}_{k,\ell}f(x) = \sum_{\ell = 0}^{\infty} \mathcal{H}_{k,\ell} f(x)$. To do so, we carry out a Littlewood--Paley decomposition. For each $\ell \ge 0$, apply
\[ 
\mathcal{H}_{k,\ell} f = \mathcal{H}_{k,\ell} S_{k-\ell} f + \mathcal{H}_{k,\ell} (f - S_{k-\ell} f).
\]
(see Section~\ref{sect:prelim} for the precise definition of $S_k f$). Notice that $S_{k-\ell} f$ is approximately constant at the physical scale $2^{-k+\ell}$. Thus, $\mathcal{H}_{k,\ell} S_{k-\ell} f$ is almost just a multiple of $S_{k-\ell} f$. This motivates us to further decompose
\[
\mathcal{H}_{k,\ell} S_{k-\ell} f = c_{\ell} S_{k-\ell} f + \left(\mathcal{H}_{k,\ell} S_{k-\ell} f -  c_{\ell} S_{k-\ell} f \right)
\]
where
\begin{equation} \label{eq:c_alpha_ell_def}
c_0 := \int_{|t| \leq 1} (e^{i |t|^{\alpha}} - 1) K(t)dt \quad \text{ and } \quad c_{\ell} := \int_{\frac{1}{2} < |t| \leq 1} e^{i 2^{\ell \alpha} |t|^{\alpha}} K(t)dt \quad \text{ for }\ell \ge 1
\end{equation}
are constants. Here we choose the constants $c_0$ and $c_{\ell}$ as such because  $K$ is assumed to be homogeneous. Hence
\[
\sum_{\ell =0}^{\infty}\mathcal{H}_{k,\ell } f(x)=
\sum_{\ell = 0}^{\infty} c_{\ell} S_{k-\ell} f 
+ \sum_{\ell=0}^{\infty} \left(\mathcal{H}_{k,\ell} S_{k-\ell} f 
-  c_{\ell} S_{k-\ell} f \right) + \sum_{\ell=0}^{\infty} \mathcal{H}_{k,\ell} (f - S_{k-\ell} f).
\]
Since a simple integration-by-parts argument shows that $|c_{\ell}| = O(2^{-\alpha \ell})$, the contribution from the first term to the desired jump norm can be controlled using (\ref{eq:clSk_jump}). To handle the latter two terms we use a square function. It suffices to show that
\begin{equation} \label{eq:sqfcna}
\sum_{\ell = 0}^{\infty} \left \| \left ( \sum_{k \in \Z}  |\mathcal{H}_{k,\ell} S_{k-\ell} f 
-  c_{\ell} S_{k-\ell} f |^2 \right )^{1/2} \right \|_{L^p(\R^n)} \lesssim \|f\|_{L^p(\R^n)}
\end{equation}
and
\begin{equation} \label{eq:sqfcnb}
\sum_{\ell = 0}^{\infty} \left \| \left ( \sum_{k \in \Z} | \mathcal{H}_{k,\ell} (f - S_{k-\ell} f)|^2 \right )^{1/2} \right \|_{L^p(\R^n)} \lesssim \|f\|_{L^p(\R^n)}
\end{equation}
since the square functions dominate the desired jump norms pointwisely.
To establish these estimates we apply a finer frequency decomposition. Let 
\[
\Delta(x) := 2^n S(2x) - S(x) \quad \text{ and } \quad \Delta_k(x) := 2^{kn} \Delta(2^k x)
\]
and write $\Delta_k f := f*\Delta_k$ so that
\[
S_{k-\ell} f = \sum_{j=1}^{\infty} \Delta_{k-\ell-j} f \quad \text{ and } \quad f - S_{k-\ell} f = \sum_{j=0}^{\infty} \Delta_{k-\ell+j} f.
\]
By the triangle inequality, to prove \eqref{eq:sqfcna} and \eqref{eq:sqfcnb}, it suffices to prove the existence of some constant $\gamma > 0$, such that
\begin{equation} \label{eq:HklSkdiff_sq_fnc_Delta}
\left\| \left( \sum_{k \in \Z} \left|\mathcal{H}_{k,\ell} \Delta_{k-\ell-j} f - c_{\ell} \Delta_{k-\ell-j} f  \right|^2 \right)^{1/2} \right\|_{L^p(\R^n)} \lesssim 2^{-\gamma(j+\ell)} \|f\|_{L^p(\R^n)}
\end{equation}
and
\begin{equation} \label{eq:Hklfdiff_sq_fnc_Delta}
\left\| \left( \sum_{k \in \Z} \left| \mathcal{H}_{k,\ell} \Delta_{k-\ell+j} f \right|^2 \right)^{1/2} \right\|_{L^p(\R^n)} \lesssim 2^{-\gamma (j+\ell)} \|f\|_{L^p(\R^n)}
\end{equation}
for every $j, \ell \geq 0$ and every $1 < p < \infty$. Throughout the paper, we use $\gamma$ to denote a positive real number that might vary form line to line, if not otherwise stated.\\

Now each of these estimates \eqref{eq:HklSkdiff_sq_fnc_Delta} and \eqref{eq:Hklfdiff_sq_fnc_Delta} holds for $1 < p < \infty$ without the small factors on the right, since  
$|\mathcal{H}_{k,\ell} f| \lesssim M f$ where $M$ is the Hardy-Littlewood maximal operator on $\R^n$, allowing us to invoke the Fefferman-Stein vector-valued inequality for the maximal function \cite[Chapter II.1]{MR1232192}. Hence by real interpolation, it suffices to prove the case $p = 2$. To do so, fix $\alpha > 0$, $\alpha \ne 1$ and $\ell \in \N$. Let $m_{\ell}(\xi)$ be the multiplier defined by
\[
\begin{split}
m_{0}(\xi) &:= \int_{|t| \leq 1} (e^{i|t|^{\alpha}}-1) e^{-it \cdot \xi} K(t)dt \\
m_{\ell}(\xi) &:= \int_{\frac{1}{2} < |t| \leq 1} e^{i  |2^{\ell} t|^{\alpha}} e^{-it \cdot \xi} K(t)dt \text{ for } \ell \ge 1.
\end{split}
\]
Let $\tilde{m}_{\ell}(\xi)$ be the multiplier defined by
\[
\begin{split}
& \tilde{m}_{0}(\xi) := \int_{|t| \leq 1} (e^{i|t|^{\alpha}}-1) (e^{-it \cdot \xi}-1) K(t)dt\\ 
& \tilde{m}_{\ell}(\xi) := \int_{\frac{1}{2} < |t| \leq 1} e^{i |2^{\ell} t|^{\alpha}} (e^{-it \cdot \xi}-1) K(t)dt \quad \text{ for } \ell \ge 1.
\end{split}
\]
Since $K$ is assumed to be homogeneous, for $\ell \geq 0$ the multiplier for $\mathcal{H}_{k,\ell}$ is $m_{\ell}(2^{-k+\ell}\xi)$. It follows that for $\ell \geq 0$ the multiplier for $\mathcal{H}_{k,\ell} - c_{\ell}$ is $\tilde{m}(2^{-k+\ell}\xi)$. Then \eqref{eq:HklSkdiff_sq_fnc_Delta} and \eqref{eq:Hklfdiff_sq_fnc_Delta} with $p = 2$ follows from the following pointwise estimates for multipliers:
\begin{equation}\label{170713e3.12a}
\left( \sum_{k \in \Z} |\widehat{\Delta}(2^{-k+\ell+j} \xi) \tilde{m}_{\ell}(2^{-k+\ell} \xi)|^2 \right)^{\frac{1}{2}}+ \left( \sum_{k \in \Z} |\widehat{\Delta}(2^{-k+\ell-j} \xi) m_{\ell}(2^{-k+\ell} \xi)|^2 \right)^{\frac{1}{2}} \lesssim 2^{-\gamma ( \ell +j)}.
\end{equation}
We need the following lemma, which is a consequence of the van der Corput lemma (details omitted): 

\begin{lem} \label{lem:ml}
We have \begin{equation}\label{170713e3.13a}
|m_{\ell}(\xi)| \lesssim \min \{2^{-\gamma \ell}, 2^{\alpha \ell}|\xi|^{-\gamma}\} \quad \text{for all $\xi \in \R$}.
\end{equation}
In particular,
\begin{equation}\label{170713e3.14a}
|m_{\ell}(\xi)| \lesssim (2^{-\gamma \ell}\cdot 2^{\alpha \ell}|\xi|^{-\gamma})^{\frac{1}{2}} \quad \text{for all $\xi \in \R$}.
\end{equation}
We also have 
\[
|\tilde{m}_{\ell}(\xi)| \lesssim 
\begin{cases} 
\min\{2^{-\gamma \ell}, |\xi|\} \lesssim 2^{-\frac{\gamma \ell}{2}} |\xi|^{\frac{1}{2}} \quad & \text{for $|\xi| \leq 1$}, \\
 1 \quad & \text{for $|\xi| \geq 1$}.
\end{cases}
\]
\end{lem}

We are ready to prove \eqref{170713e3.12a}. The estimate is invariant upon replacing $\xi$ by $2\xi$, hence we only need to prove it when $|\xi| \simeq 1$. First consider the first term on the left hand side of \eqref{170713e3.12a}. When $k\le 0$, we bound $|\tilde{m}_{\ell} (2^{-k+\ell}\xi)|\lesssim 1$ and $|\widehat{\Delta}(2^{-k+ \ell +j}\xi)|\lesssim 2^{-10(-k+\ell+j)}$. Summing over $k\le 0$, we obtain $2^{-10( \ell +j)}$.

When $k\ge 0$, we bound $|\tilde{m}_{\ell} (2^{-k+\ell}\xi)|\lesssim 2^{-\frac{\gamma \ell}{2}}2^{-\frac{k}{2}+\frac{\ell}{2}}$ and 
\[
|\widehat{\Delta}(2^{-k+\ell+j}\xi)| \lesssim
  \begin{cases} 
       2^{-10(-k+\ell +j)}   \hfill & \text{ if $0\le k\le \ell + j$} \\
       2^{-k+\ell + j} \hfill & \text{ if $k\ge \ell +j$} \\
  \end{cases}
\]
Summing over $k\ge 0$, we obtain $2^{-\gamma(\ell + j)}$ for some $\gamma>0$. This finishes the proof of the first half of \eqref{170713e3.12a}.\\

Next we turn to the second term on the left hand side of \eqref{170713e3.12a}. What we need to prove can also be written as 
\begin{equation}\label{170713e3.17a}
\left( \sum_{k \in \Z} |\widehat{\Delta}(2^{k} \xi) m_{\ell}(2^{k +j} \xi)|^2 \right)^{\frac{1}{2}} \lesssim 2^{-\gamma (\ell +j)} \text{ for } |\xi| \simeq 1.
\end{equation}
We work on two different cases. Let $C_{\alpha}>0$ be a sufficiently large constant. Assume that we are in the case $j\ge C_{\alpha } \ell $. We bound the left hand side of \eqref{170713e3.17a} by 
\[
\sum_{k\ge 0} 2^{-10 k} 2^{\alpha \ell} 2^{-\gamma k -\gamma j} + \sum_{k<0} 2^k (2^{\alpha \ell}\cdot 2^{-\gamma \ell} 2^{-\gamma k-\gamma j})^{1/2}  \lesssim 2^{-\gamma (\ell +j)}.
\]
Here for the case $k\ge 0$ we applied \eqref{170713e3.13a}, and for the case $k<0$ we applied \eqref{170713e3.14a}. 

Finally, we assume that $0\le j\le C_{\alpha} \ell$. We bound the left hand side of \eqref{170713e3.17a} by 
\[
\sum_{k\ge 0} 2^{-10 k} 2^{-\gamma \ell} + \sum_{k<0}2^k 2^{-\gamma \ell} \lesssim 2^{-\gamma (\ell + j)}.
\]
Here in both cases $k\ge 0$ and $k<0$ we applied \eqref{170713e3.13a}.

\section{Short jump estimates for large $p$}\label{section:short-jump-large}
 
We are now going to start the proof of Proposition \ref{main2}. Recall that by interpolation, we only need to establish Proposition \ref{main2} when $p \in (\frac{2n}{2n-1},\infty)$ and $r \in (2,\infty)$ (see discussion following Proposition~\ref{main2}).
In this section we will do so for all sufficiently large values of $p$. 
More precisely, let $\alpha > 1$, $\mathcal{H}^{(u)}$ be as in (\ref{eq:Hu_def}), and let $V^{r}_{ j}{\mathcal{H} f}(x)
= V^{r} \{\mathcal{H}^{(u)} f (x) \colon u \in [2^{j\alpha}, 2^{(j+1)\alpha}]\}$. We prove
\begin{equation} \label{eq:mainest_Sect4}
\left\| \left(\sum_{j\in \Z} |V_j^{r} (\mathcal{H}f)|^{r} \right)^{\frac{1}{r}}\right\|_p \lesssim \|f\|_p,
\end{equation}
whenever
\begin{equation} \label{eq:cond_n1}
p \in (2,\infty), \quad n = 1, \quad r \in (2, \infty) 
\end{equation}
or 
\begin{equation} \label{eq:cond_n2}
p \in \left(2+\frac{4}{n}, \infty \right), \quad n \geq 2, \quad r \in [2, \infty).
\end{equation}  
This proves Proposition~\ref{main2} when $n = 1$. In the next section, we extend \eqref{eq:mainest_Sect4} to all $p \in ( \frac{2n}{2n-1},\infty)$ when $n \geq 2$, $r \in [2,\infty)$. That would complete the proof of Proposition~\ref{main2} when $n \geq 2$.

\subsection{Main tool: A square function estimate for the semigroup $e^{it(-\Delta)^{\lambda/2}}$} The main input to our proof of \eqref{eq:mainest_Sect4} under conditions \eqref{eq:cond_n1} or \eqref{eq:cond_n2} is a square function estimate, due to Lee, Rogers and Seeger \cite{MR2946085}:

\begin{prop}[Lee, Rogers and Seeger \cite{MR2946085}]\label{LRSsquare}
\begin{enumerate}
\item Let $n = 1$, $p \in [2,\infty)$ and $\lambda>1$. Then for any compact time interval $I$, 
\[
\left\|\left(\int_{I} \left| \int_{\R}e^{ix\xi}\hat{f}(\xi)e^{it|\xi|^{\lambda}} d\xi\right|^2 dt\right)^{1/2}\right\|_{L^p(\R)} \lesssim \|f\|_{L^{p}(\R)}.
\]
\item Let $n\ge 2$, $p \in (\frac{2(n+2)}{n},\infty)$ and $\lambda >1$. Then for any compact time interval $I$, 
\[
\left\|\left(\int_{I} \left| \int_{\R^n}e^{ix \cdot \xi}\hat{f}(\xi)e^{it|\xi|^{\lambda}} d\xi\right|^2 dt\right)^{1/2}\right\|_{L^p(\R^n)} \lesssim \|f\|_{W^{\beta,p}(\R^n)},
\]
with $\frac{ \beta }{ \lambda }=n(\frac{1}{2}-\frac{1}{p})-\frac{1}{2}$.
\end{enumerate}
\end{prop}
We will apply the above estimates with $\lambda = \alpha' := \frac{\alpha}{\alpha-1}$ (remember $\alpha > 1$).
Recall that we are interested in the variation of $\mathcal{H}^{(u)} f(x)$ where $u$ is restricted to the range $[2^{j \alpha}, 2^{(j+1) \alpha}]$ for some $j \in \Z$. To estimate this, we decompose the kernel $e^{i u |t|^{\alpha}} K(t)$ into a part where oscillation plays no role, and a part where the oscillation becomes important. More precisely, for $\ell \in \Z$, let
\begin{equation} \label{eq:Hul} 
\mathcal{H}^{(u)}_{\ell} f(x) := \int_{\R^n} f(x-t) e^{i u |t|^{\alpha}} \varphi_{\ell}(t) K(t) dt,
\end{equation}
where $\varphi_{\ell}(t) = \varphi_0(2^{-\ell} t)$ and $\varphi_0$ is radial, smooth and compactly supported on an annulus $\{|t| \simeq 1\}$ so that for $t \ne 0$, $\sum_{\ell \in \Z} \varphi_{\ell}(t) = 1$. When $u \simeq 2^{j \alpha}$, $|t| \simeq 2^{\ell-j}$, the phase $e^{i u |t|^{\alpha}}$ in (\ref{eq:Hul}) is approximately 1 precisely when $\ell < 0$. Thus, it makes sense to decompose
\begin{equation} \label{eq:decompl}
\mathcal{H}^{(u)} f(x) = \sum_{\ell \in \Z} \mathcal{H}^{(u)}_{\ell-j} f(x)
\end{equation}
and expect that the terms $\ell < 0$ in the above sum are essentially non-oscillatory. 

It suffices to show that
\begin{equation} \label{eq:mainest}
\sum_{\ell \in \Z} \left\| \left(\sum_{j\in \Z} |V_j^{r} \mathcal{H}^{(u)}_{\ell-j} f |^{r} \right)^{\frac{1}{r}}\right\|_p \lesssim \|f\|_p.
\end{equation}
To do so, we introduce a Littlewood--Paley decomposition in the $x$ variable. Let $P_k$ be a multiplier operator defined by $\widehat{P_k f}(\xi) =   \psi  (2^{-k} \xi) \widehat{f}(\xi)$, where $  \psi  $ is a smooth function with compact support on the annulus $1/2 \leq |\xi| \leq 2$, so that for $\xi \ne 0$, $\sum_{k \in \Z}   \psi  (2^{-k} \xi) = 1$. We further decompose 
\begin{equation} \label{eq:decompk}
\mathcal{H}^{(u)}_{\ell-j} f(x) = \sum_{k \in \Z} \mathcal{H}^{(u)}_{\ell-j} P_{j+k} f(x).
\end{equation}
We will estimate
\begin{equation} \label{eq:mainest1}
\left\| \|V_j^{r} \mathcal{H}^{(u)}_{\ell-j} P_{j+k} f \|_{\ell^r_j} \right\|_p 
\end{equation}
for each $k, \ell \in \Z$, and sum the estimates at the end. (Hereafter, for compactness of notations, we write $\ell^r_j$ for the $\ell^r$ norm over all $j \in \Z$.)

\subsection{Estimates for $\ell \leq -\frac{k}{2(\alpha+1)}$: Bounding the $V^r_j$ norm by the $\dot{W}^{1,1}$ norm} \label{sect4.2}

First there are two simple estimates for (\ref{eq:mainest1}). One way to estimate (\ref{eq:mainest1}) is to bound the $V^{r}_j$ norm by the $V^1_j$ norm, which in turn is bounded by the $\dot{W}^{1,1}$ norm on the $u$ interval $[2^{j\alpha}, 2^{(j+1)\alpha}]$. We get 
\[
V^{r}_j \mathcal{H}^{(u)}_{\ell-j} P_{j+k} f(x) 
\lesssim \int_{|u| \simeq 2^{j\alpha}} 2^{(\ell - j)\alpha} \int_{|t| \simeq 2^{\ell-j}} |P_{j+k} f(x-t)| |K(t)| dt du
\lesssim 2^{\ell \alpha} M P_{j+k} f(x),
\]
where $M$ is the Hardy-Littlewood maximal function, so by the Fefferman-Stein inequality and Littlewood--Paley inequality, we have
\begin{equation} \label{est1}
\| \| V^{r}_j \mathcal{H}^{(u)}_{\ell-j} P_{j+k} f(x) \|_{\ell^{r}_j} \|_{L^p_x} 
\lesssim 2^{\ell \alpha} \|f\|_{L^p}, \quad 1 < p < \infty.
\end{equation} 

For the second simple estimate, recall that $\int_{|t|=R} K(t) d\sigma(t) = 0$ for all $R \in (0, \infty)$. Since $\varphi$ was chosen to be radial, we have
\[
\int_{\R^n}  e^{i u |t|^{\alpha}} \varphi_{\ell-j}(t) K(t) dt = 0.
\]
Thus, in computing $V^{r}_j \mathcal{H}^{(u)}_{\ell-j} P_{j+k} f(x)$, we could have instead computed the $V^{r}_j$ norm of 
\[
\mathcal{H}^{(u)}_{\ell-j} P_{j+k} f(x) - P_{j+k}f(x) \int_{\R^n}  e^{i u |t|^{\alpha}} \varphi_{\ell-j}(t) K(t) dt.
\]
This expression is equal to
\begin{align}
\int_{\R^n} [P_{j+k}f(x-t) - P_{j+k}f(x)] e^{iu |t|^{\alpha}} \varphi_{\ell-j}(t) K(t) dt \notag \end{align}
The variational norm of this expression is controlled by its $\dot{W}^{1,1}$ norm in the $u$ interval $[2^{j\alpha}, 2^{(j+1)\alpha}]$, which in turn is controlled by
\[
2^{j+k} 2^{\ell-j} 2^{\ell \alpha} M \tilde{P}_{j+k} f(x)
\]	
where $\tilde{P}_{j+k}$ is a variant of the Littlewood--Paley projection $P_{j+k}$, so arguing as before, we see that
\begin{equation} \label{est2}
\| \| V^{r}_j \mathcal{H}^{(u)}_{\ell-j} P_{j+k} f(x) \|_{\ell^{r}_j} \|_{L^p_x} \lesssim 2^{\ell+k} 2^{\ell \alpha} \|f\|_{L^p}, \quad 1 < p < \infty.
\end{equation} 

We can sum \eqref{est2} over all pairs $(k,\ell)$ with $\ell \leq -\frac{k}{2(\alpha+1)}$ and $k \leq 0$. We can also sum \eqref{est1} over all $(k,\ell)$ with $\ell \leq -\frac{k}{2(\alpha+1)}$ and $k \geq 0$. Thus, it remains to bound \eqref{eq:mainest1} when 
\begin{equation} \label{eq:klcaseremain}
\ell > -\frac{k}{2(\alpha+1)}
\end{equation} 
and sum over all such pairs of $(k,\ell)$.

\subsection{Estimates for $\ell > -\frac{k}{2(\alpha+1)}$: Division into 3 cases} \label{sect4.3}

First we look at $\mathcal{H}^{(u)}_{\ell-j} P_{j+k} f (x)$ in terms of its multiplier:
\[
\mathcal{H}^{(u)}_{\ell-j} P_{j+k} f (x) = \frac{1}{(2 \pi)^n} \int_{\R^n} \widehat{f}(\xi) \left(   \psi  (2^{-j-k} \xi) \int_{\R^n} e^{-i t \cdot \xi} e^{i u |t|^{\alpha}} \varphi_{\ell-j}(t) K(t) dt \right) e^{i x \cdot \xi} dx.
\] 
The multiplier is an oscillatory integral in $t$ with phase $\phi(t) = - t \cdot \xi + u |t|^{\alpha}$, which (assuming $|u| \simeq 2^{\alpha j}$ and $|\xi| \simeq 2^{j+k}$) has a critical point in the annulus $\{|t| \simeq 2^{\ell-j}\}$ if and only if $2^{k+\ell} \simeq 2^{\ell \alpha}$, that is, if and only if $k = \ell (\alpha-1) + O(1)$. In that case, using stationary phase (see, for example, \cite[Chapter VIII.5.7]{MR1232192} or \cite[Theorem 1.2.1]{MR1205579}), the multiplier can be written as
\begin{equation} \label{eq:multiplier_expansion}
\psi(2^{-j-k} \xi) \left( e^{i c_{\alpha} (2^{-j\alpha} u)^{-\frac{1}{\alpha-1}} (2^{-j} |\xi|)^{\alpha'}} a(2^{\ell} 2^{-j} \xi, 2^{\ell \alpha} 2^{-j\alpha} u) + e(2^{\ell} 2^{-j} \xi, 2^{\ell \alpha} 2^{-j\alpha} u) \right)
\end{equation} 
where $\alpha' = \frac{\alpha}{\alpha-1}$, $c_{\alpha} = \frac{\alpha-1}{\alpha^{\alpha'}}$, $a \in S^{-n/2}(\R^{n+1})$ and $e \in S^{-\infty}(\R^{n+1})$. 
If there were no critical points in the annulus $\{|t| \simeq 2^{\ell-j}\}$, then the multiplier is simply 
\begin{equation} \label{eq:multiplier_expansion2}
\psi(2^{-j-k} \xi)  e(2^{\ell} 2^{-j} \xi, 2^{\ell \alpha} 2^{-j\alpha} u).
\end{equation}  
(In the above, by $a \in S^{-n/2}(\R^{n+1})$ we mean
\[
|\partial_{\xi}^{\alpha'} \partial_u^{\alpha''} a(\xi,u)| \lesssim_{\alpha} (1+|\xi|+|u|)^{-n/2-|\alpha|}
\] 
for every multiindex $\alpha = (\alpha',\alpha'') \in \mathbb{Z}_{\geq 0}^{n+1}$, and by $e \in \in S^{-\infty}(\R^{n+1})$ we mean
\[
|\partial_{\xi}^{\alpha'} \partial_u^{\alpha''} e(\xi,u)| \lesssim_{N,\alpha} (1+|\xi|+|u|)^{-N-|\alpha|}
\] 
for any positive integers $N$ and any multiindex $\alpha$.)

The above motivates us to consider three cases separately (under our earlier standing assumption \eqref{eq:klcaseremain}): \\

\noindent{\textbf{Case 1.}} $\ell \geq 0$, $k = \ell (\alpha-1) + O(1)$;  \\ 
\noindent{\textbf{Case 2.}} $k > \ell (\alpha - 1) + C$ for some $C > 0$; \\ 
\noindent{\textbf{Case 3.}} $k < \ell (\alpha - 1) - C$ for some $C > 0$. \\

\subsection{Estimates in Case 1} \label{sect4.4}
Now we consider Case 1. Our goal is to bound \eqref{eq:mainest1} given $k$ and $\ell$ as in Case 1. We proceed in a few steps.

\subsubsection{Application of Plancherel--P\'{o}lya}

First we will essentially show that if $r \in [2,\infty)$, then
\begin{equation} \label{eq:Lplqlq}
\left\| \|V_j^{r} \mathcal{H}^{(u)}_{\ell-j} P_{j+k} f \|_{\ell^r_j} \right\|_{L^p_x}  \lesssim 
2^{\frac{\ell \alpha}{q}}  \| \| \| \chi(u) \mathcal{H}^{(2^{j \alpha} u)}_{\ell-j} P_{j+k} f (x) \|_{L^q_u} \|_{\ell^q_j} \|_{L^p_x}
\end{equation}
for any $q \in [2,r]$ and any $p \in [1,\infty]$; here $\chi(u)$ is a smooth function with compact support on $[1/2,2^{\alpha+1}]$ that is identically equal to 1 on $[1,2^{\alpha}]$. Indeed, when $n \geq 2$ (and $p, r$ are as in (\ref{eq:cond_n2})), we will only need (\ref{eq:Lplqlq}) for $q = 2$. But for $n = 1$ (and $p, r$ as in (\ref{eq:cond_n1})), we will need (\ref{eq:Lplqlq}) for both $q = 2$ and $q = r$. We will see that this is the case after we prove (\ref{eq:Lplqlq}).

To prove (\ref{eq:Lplqlq}), let us temporarily write $g = P_{j+k} f$. As a function of $u$, $\mathcal{H}^{(u)}_{\ell-j} g$ has frequency morally supported on the annulus of size $\simeq 2^{(\ell-j)\alpha}$ centered at the origin. Thus, we introduce Littlewood--Paley projections in the $u$ variable (denoted by $P^{(2)}$ so that $P^{(2)}_{(\ell-j) \alpha}$ is projection onto frequency $\simeq 2^{(\ell-j) \alpha}$) and estimate
\begin{equation} \label{lemma1est}
\begin{split}
&|V^r_j \mathcal{H}^{(u)}_{\ell-j} g(x)| \\
\leq & |V^r (P^{(2)}_{\leq (\ell-j)\alpha} [\chi(2^{-j\alpha} u) \mathcal{H}^{(u)}_{\ell-j} g(x)])| + \sum_{k = 1}^{\infty} |V^r (P^{(2)}_{(\ell-j+k)\alpha} [\chi(2^{-j\alpha} u) \mathcal{H}^{(u)}_{\ell-j} g(x)])|.
\end{split}
\end{equation}
(Here $P^{(2)}_{\leq (\ell-j)\alpha} := \sum_{k \leq \ell-j} P^{(2)}_{k \alpha}$.) 

The first term on the right hand side of (\ref{lemma1est}) is the main term, and can be estimated using Proposition~\ref{PPinequality}. In particular, it is bounded by
\[
2^{\frac{(\ell-j) \alpha}{q}} \| \chi(2^{-j\alpha} u) \mathcal{H}^{(u)}_{\ell-j} g(x) \|_{L^q_u}
\]
(recall $q \in [2,r]$).
By changing variable in $u$, this is just
\[
2^{\frac{\ell \alpha}{q}} \| \chi(u) \mathcal{H}^{(2^{j \alpha} u)}_{\ell-j} g(x) \|_{L^q_u}.
\]
Hence the contribution of the first term of \eqref{lemma1est} to the left hand side of \eqref{eq:Lplqlq} is bounded by
\[ 
2^{\frac{\ell \alpha}{q}} \| \| \| \chi(u) \mathcal{H}^{(2^{j \alpha} u)}_{\ell-j} P_{j+k} f(x) \|_{L^q_u} \|_{\ell^r_j} \|_{L^p_x}.
\]
Since $r \geq q$, we have $\ell^{r}$ norm bounded by $\ell^q$ norm, hence the above is bounded by the right hand side of \eqref{eq:Lplqlq}. 

On the other hand, for the second term on the right hand side of (\ref{lemma1est}), since $k > -C$, one can integrate by parts in $u$, using the fact that the multiplier for $P^{(2)}_{(\ell-j+k)\alpha}$ vanishes to infinite order at 0, and obtain
\begin{equation} \label{eq:errorgain}
|P^{(2)}_{(\ell-j+k)\alpha} [\chi(2^{-j\alpha} u) \mathcal{H}^{(u)}_{\ell-j} g(x)]|
\lesssim_N 2^{-k\alpha N} \tilde{P^{(2)}}_{(\ell-j+k)\alpha} [\chi(2^{-j\alpha} u) \tilde{\mathcal{H}}^{(u)}_{\ell-j} g(x)]
\end{equation}
for any positive integer $N$, where $\tilde{P^{(2)}}$ is a Littlewood--Paley projection similar to $P^{(2)}$, and $\tilde{\mathcal{H}}^{(u)}_{\ell-j}$ is the same as $\mathcal{H}^{(u)}_{\ell-j}$ defined in (\ref{eq:Hul}), except that the cutoff $\varphi$ is replaced by a smooth multiple $\tilde{\varphi}$ of $\varphi$. Hence by repeating the above argument, and summing over $k$ using the additional convergence factors $2^{-k\alpha N}$ that we gained in \eqref{eq:errorgain}, the contribution of the second term of \eqref{lemma1est} to the left hand side of \eqref{eq:Lplqlq} is bounded by
\begin{equation} \label{eq:Lplqlq_B}
2^{\frac{\ell \alpha}{q}}  \| \| \| \chi(u) \tilde{\mathcal{H}}^{(2^{j \alpha} u)}_{\ell-j} P_{j+k} f(x) \|_{L^q_u} \|_{\ell^q_j} \|_{L^p_x}.
\end{equation}
Since $\tilde{\mathcal{H}}$ and $\mathcal{H}$ satisfies the same estimates, we will not distinguish the two, and declare that we can also bound \eqref{eq:Lplqlq_B} once we can bound the right hand side of \eqref{eq:Lplqlq}.

\subsubsection{Application of the square function estimate}

Now fix $k,\ell$ as in Case 1. In other words, fix $k,\ell \geq 0$ with $k = \ell(\alpha-1) + O(1)$. We will try to bound the right hand side of (\ref{eq:Lplqlq}) when $q = 2$. The multiplier for $\mathcal{H}^{(2^{j \alpha} u)}_{\ell-j} P_{j+k} f$ is given by \eqref{eq:multiplier_expansion} with $u$ replaced by $2^{j\alpha}u$. For $u \in \R$, let $\tilde{m}_u(\xi)$ be the multiplier  
\begin{equation} \label{eq:tildemu_def}
\tilde{m}_u(\xi) = \chi(u) \psi(2^{-k} \xi) \left( e^{i c_{\alpha} u^{-\frac{1}{\alpha-1}} |\xi|^{\alpha'}} a(2^{\ell} \xi, 2^{\ell \alpha} u) + 
e(2^{\ell} \xi, 2^{\ell \alpha} u) \right)
\end{equation}
where $a \in S^{-n/2}(\R^{n+1})$, $e \in S^{-\infty}(\R^{n+1})$ are as in \eqref{eq:multiplier_expansion}. Then the multiplier of the operator $\chi(u) \mathcal{H}^{(2^{j \alpha} u)}_{\ell-j} P_{j+k}$ is precisely $\tilde{m}_u(2^{-j} \xi)$.
Now expand $\chi(u) a(2^{\ell} \xi, 2^{\ell \alpha} u)$ in Fourier series in $u$: let $c$ be a small enough constant depending on $\alpha$ so that the support of $\chi(u)$ is contained in $[0,c^{-1}]$. Using the smoothness in the variable $u$, we get
\[
\chi(u) a(2^{\ell} \xi, 2^{\ell \alpha} u)
= \sum_{\kappa \in c \Z} (1+|\kappa|)^{-2} a_{\kappa}(2^{\ell} \xi) e^{i \kappa u}
\]
for $u \in [0,c^{-1}]$, where $a_{\kappa} \in S^{-n/2}(\R^n)$ uniformly for every $\kappa \in c \Z$. Similarly, expand $\chi(u) e(2^{\ell} \xi, 2^{\ell \alpha} u)$ in Fourier series in $u$:
\[
\chi(u) e(2^{\ell} \xi, 2^{\ell \alpha} u)
= \sum_{\kappa \in c \Z} (1+|\kappa|)^{-2} e_{\kappa}(2^{\ell} \xi) e^{i \kappa u}
\]
for $u \in [0,c^{-1}]$, where $e_{\kappa} \in S^{-\infty}(\R^n)$ uniformly for every $\kappa \in c \Z$.
This shows 
\[ 
\tilde{m}_u(\xi) = \sum_{\kappa \in c \Z} (1+|\kappa|)^{-2} e^{i \kappa u} \psi(2^{-k} \xi) \left(  a_{\kappa}(2^{\ell} \xi) e^{i c_{\alpha} u^{-\frac{1}{\alpha-1}} |\xi|^{\alpha'}} + e_{\kappa}(2^{\ell} \xi) \right)
\]
for $u \in [0,c^{-1}]$. Temporarily let $g$ be the function such that $\widehat{g}(\xi) = \widehat{f}(\xi) \psi(2^{-k} \xi) a_{\kappa}(2^{\ell} \xi)$; note that when $k \geq 0$, $\|g\|_{L^p_{\beta}(\R^n)} \lesssim 2^{k \beta} 2^{-(k+\ell) n/2} \|f\|_{L^p(\R^n)}$ by the H\"ormander--Mikhlin multiplier theorem, with an implicit constant independent of $\kappa$. This is further bounded by $2^{\ell (\alpha-1) \beta} 2^{-\ell \alpha n/2} \|f\|_{L^p(\R^n)}$ since we are in Case 1, where $k = \ell(\alpha-1) + O(1)$. We apply Proposition~\ref{LRSsquare} with $g$ in place of $f$ and obtain
\begin{equation} \label{eq:square_fcn_Sect4}
\begin{split}
& \left\|  \left\|  \int_{\R^n} e^{i x \cdot \xi} \widehat{f}(\xi) \psi(2^{-k} \xi) a_{\kappa}(2^{\ell} \xi) e^{i c_{\alpha} u^{-\frac{1}{\alpha-1}} |\xi|^{\alpha'}} d\xi \right\|_{L^2_u [0,c^{-1}]} \right\|_{L^p(\R^n)} \\
& \lesssim 
\begin{cases}
2^{-\ell \alpha/2} \|f\|_{L^p(\R)} &\quad \text{if $p \in [2,\infty)$ and $n = 1$} \\
2^{\ell \alpha \left[n \left( \frac{1}{2}-\frac{1}{p} \right) - \frac{1}{2} \right]} 2^{-\ell \alpha n/2} \|f\|_{L^p(\R^n)} 
&\quad \text{if $p \in ( 2+\frac{4}{n}, \infty) $ and $n \geq 2$}.
\end{cases}
\end{split}
\end{equation}
We get a better decay if $a_{\kappa}(2^{\ell} \xi) e^{i c_{\alpha} u^{-\frac{1}{\alpha-1}} |\xi|^{\alpha'}}$ above is replaced by $e_{\kappa}(2^{\ell} \xi)$. Summing over $\kappa$, and simplifying the exponent in case $n \geq 2$, we get
\[ 
\begin{split}
& \left\|  \left\|  \int_{\R^n} e^{i x \cdot \xi} \widehat{f}(\xi) \tilde{m}_u(\xi) d\xi \right\|_{L^2_u} \right\|_{L^p(\R^n)} \lesssim 
\begin{cases}
2^{-\ell \alpha/2} \|f\|_{L^p(\R)} &\quad \text{if $p \in [2,\infty)$ and $n = 1$} \\
2^{-\ell \alpha/2} 2^{-\ell \alpha n/p}  \|f\|_{L^p(\R^n)}
&\quad \text{if $p \in ( 2+\frac{4}{n}, \infty)$ and $n \geq 2$}.
\end{cases}
\end{split}
\]
But recall that the multiplier of the operator $\chi(u) \mathcal{H}^{(2^{j \alpha} u)}_{\ell-j} P_{j+k}$ is precisely $\tilde{m}_u(2^{-j} \xi)$. By scale invariance, we have
\begin{equation} \label{eq:square_fcn_3}
\begin{split}
& \left\|  \left\|  \chi(u) \mathcal{H}^{(2^{j \alpha} u)}_{\ell-j} P_{j+k} f \right\|_{L^2_u} \right\|_{L^p(\R^n)}  \lesssim 
\begin{cases}
2^{-\ell \alpha/2} \|f\|_{L^p(\R)} &\quad \text{if $p \in [2,\infty)$ and $n = 1$} \\
2^{-\ell \alpha/2} 2^{-\ell \alpha n/p}  \|f\|_{L^p(\R^n)} 
&\quad \text{if $p \in ( 2+\frac{4}{n}, \infty)$ and $n \geq 2$} 
\end{cases}
\end{split}
\end{equation}
for all $j \in \Z$, where the implicit constants are independent of $j$. (The Fourier series expansion used to remove the dependence on $u$ are very reminiscent of the method used to prove $L^2$ boundedness of multipliers in $S^0$; see, for example, \cite[Chapter VI.2]{MR1232192}.)

Recall that our goal now is to bound the right hand side of (\ref{eq:Lplqlq}) when $q = 2$. Hence we need a vector-valued version of (\ref{eq:square_fcn_3}), where we will have an additional $\ell^2$ norm over $j \in \Z$ inside the $L^p$ norm on the left hand side of (\ref{eq:square_fcn_3}). To do so, we need a vector-valued variant of a theorem of Seeger, about multipliers with localized bounds. This will be done in the next subsection.

\subsubsection{Application of Seeger's theorem for multipliers with localized bounds} 

First we state a vector-valued variant of a theorem of Seeger, about multipliers with localized bounds:

\begin{prop} [Jones, Seeger, Wright \cite{MR2434308}, Seeger \cite{MR955772}] \label{prop:Seeger88}
Let $I \subset \R$ be a compact interval. Let $\{\tilde{m}_u(\xi) \colon u \in I \}$ be a family of Fourier multipliers on $\R^n$, each of which is compactly supported on $\{\xi \colon 1/2 \leq |\xi| \leq 2\}$, and satisfies
\[
\sup_{u \in I} |\partial_{\xi}^{\tau} \tilde{m}_u(\xi)| \leq B \quad \text{for each $0 \leq |\tau| \leq n+1$}
\]
for some constant $B$. For $u \in I$ and $j \in \Z$, write $T_{u,j}$ the multiplier operator with multiplier $\tilde{m}_u(2^{-j} \xi)$.
Fix some $p \in [2, \infty)$. Assume that there exists some constant $A$ such that
\begin{equation}  \label{eq:assump_1piece}
\sup_{j \in \Z} \left\| \| T_{u,j} f \|_{L^2(I)} \right\|_{L^s(\R^n)} \leq A \|f\|_{L^s(\R^n)} 
\end{equation}
for both $s = p$ and $s = 2$.
Then 
\[
\left\| \| \|T_{u,j} f\|_{L^2(I)} \|_{\ell^2(\Z)} \right\|_{L^p(\R^n)} \lesssim A \left| \log \left( 2 + \frac{B}{A} \right) \right|^{\frac{1}{2}-\frac{1}{p}} \|f\|_{L^p(\R^n)}.
\]
\end{prop}
This proposition was stated without proof on p.6737 of Jones, Seeger and Wright \cite{MR2434308}. It is a vector-valued analog of Theorem 1 of Seeger \cite{MR955772}, and we provide a proof of this proposition in Section \ref{180320section7} for the convenience of the reader. 

Recall that our goal is to bound the right hand side of (\ref{eq:Lplqlq}) when $q = 2$. Also recall that if $\tilde{m}_u(\xi)$ is defined as in (\ref{eq:tildemu_def}), and $T_{u,j}$ is the multiplier operator with multiplier $\tilde{m}_u(2^{-j}\xi)$ as in Proposition~\ref{prop:Seeger88}, then $T_{u,j}f$ is precisely $\chi(u) \mathcal{H}^{(2^{j \alpha} u)}_{\ell-j} P_{j+k} f$. Thus, if we could apply Proposition~\ref{prop:Seeger88}, we would obtain a bound about the right hand side of (\ref{eq:Lplqlq}) when $q = 2$. To do so we verify the hypothesis of Proposition~\ref{prop:Seeger88}. From the explicit expression (\ref{eq:tildemu_def}), we have $$\sup_{u \in I} |\partial_{\xi}^{\tau} \tilde{m}_u(\xi)| \lesssim 2^{\ell N}$$ for some large positive integer $N$, if $|\tau| \leq n+1$. The hypothesis (\ref{eq:assump_1piece}) for $s = p$ is given by (\ref{eq:square_fcn_3}), where $A$ can be chosen to be relatively small if $\ell$ is large. On the other hand, by considering the $L^{\infty}$ norm of the multipliers, we also get
\begin{equation} \label{eq:L2_mult_bound}
\left\|  \left\|  \chi(u) \mathcal{H}^{(2^{j \alpha} u)}_{\ell-j} P_{j+k} f \right\|_{L^2_u} \right\|_{L^2(\R^n)} \lesssim 
2^{-\ell \alpha n/2} \|f\|_{L^2(\R^n)} \quad \text{for all $n \geq 1$},
\end{equation}
which gives us the hypothesis (\ref{eq:assump_1piece}) for $s = 2$, where $A$ can be chosen to be relatively small if $\ell$ is large.

More precisely, suppose first $n \geq 2$ and $p \in (2+\frac{4}{n},\infty)$. Then we invoke (\ref{eq:square_fcn_3}) and (\ref{eq:L2_mult_bound}). Since $2^{-\ell \alpha n/2} \leq 2^{-\ell \alpha/2} 2^{-\ell \alpha n/p}$, we may apply Proposition~\ref{prop:Seeger88} with $A = 2^{-\ell \alpha/2} 2^{-\ell \alpha n/p} $, and $B = 2^{\ell N}$ for some large positive integer $N$ depending only on $\alpha$. Thus, if $n \geq 2$ and $p \in (2+\frac{4}{n},\infty)$, then we get
\[
\| \| \| \chi(u) \mathcal{H}^{(2^{j \alpha} u)}_{\ell-j} P_{j+k} f (x) \|_{L^2_u} \|_{\ell^2_j} \|_{L^p_x} \lesssim_{\varepsilon} 
2^{-\ell \alpha/2} 2^{-\ell \alpha n/p}  2^{\ell \varepsilon} \|f\|_{L^p(\R^n)} 
\]
for any $\varepsilon > 0$. Taking $q = 2$ in \eqref{eq:Lplqlq}, this shows that
\[
\left\| \|V_j^{r} \mathcal{H}^{(u)}_{\ell-j} P_{j+k} f \|_{\ell^r_j} \right\|_p  \lesssim_{\varepsilon} 2^{-\ell \alpha n/p} 2^{\ell \varepsilon} \|f\|_{L^p(\R^n)} \quad \text{if $n \geq 2$, $p \in (2+\frac{4}{n},\infty)$ and $r \in [2, \infty)$.}
\]
Note that the power of $2$ here is negative. So this estimate can be summed over all $\ell \geq 0$, and this gives the desired bound for \eqref{eq:mainest1} when $n \geq 2$, $p \in (2+\frac{4}{n},\infty)$ and $r \in [2,\infty)$ for $k, \ell$ as in Case 1.

On the other hand, if $n = 1$ and $p \in [2,\infty)$, then in light of (\ref{eq:square_fcn_3}) and (\ref{eq:L2_mult_bound}), we may apply Proposition~\ref{prop:Seeger88} with $A = 2^{-\ell \alpha/2}$ and $B = 2^{\ell N}$ for some large positive integer $N$ depending on $\alpha$. We obtain
\[ 
\| \| \| \chi(u) \mathcal{H}^{(2^{j \alpha} u)}_{\ell-j} P_{j+k} f (x) \|_{L^2_u} \|_{\ell^2_j} \|_{L^p_x} \lesssim_{\varepsilon}  2^{-\ell \alpha/2} 2^{\ell \varepsilon} \|f\|_{L^p(\R)} \quad \text{if $n = 1$ and $p  \in [2,\infty)$}
\]
for any $\varepsilon > 0$. Taking $q = 2$ in \eqref{eq:Lplqlq}, this shows that
\begin{equation} \label{eq:squarefcn1}
\left\| \|V_j^{r} \mathcal{H}^{(u)}_{\ell-j} P_{j+k} f \|_{\ell^r_j} \right\|_p \lesssim_{\varepsilon} 2^{\ell \varepsilon} \|f\|_{L^p(\R)} \quad \text{if $n = 1$, $p \in [2,\infty)$ and $r \in [2,\infty)$}.
\end{equation}
This is not good enough to be summed over all $\ell \geq 0$, so we need to gain a slightly better decay in $\ell$. This is achieved via the local smoothing estimate in Theorem \ref{180320thm1.6}.

\subsubsection{Application of a local smoothing estimate in dimension $n = 1$}
The goal of this subsection is to prove that 
\begin{equation} \label{eq:localsmooth2}
\left\| \|V_j^{r} \mathcal{H}^{(u)}_{\ell-j} P_{j+k} f \|_{\ell^r_j} \right\|_p  \lesssim 2^{-\ell \alpha/p} \|f\|_{L^p(\R)} \quad \text{if $n = 1$, $p = r \in (4,\infty)$}.
\end{equation}
Assume for the moment that this has been established. Interpolating \eqref{eq:localsmooth2} against \eqref{eq:squarefcn1} using complex interpolation of vector-valued $L^p$ spaces (see \cite[Theorem 5.1.2]{BL76}), we get
\begin{equation} \label{eq:eachpiece}
\begin{split}
&\left\| \|V_j^{r} \mathcal{H}^{(u)}_{\ell-j} P_{j+k} f \|_{\ell^r_j} \right\|_p  \lesssim 2^{-\gamma \ell} \|f\|_{L^p(\R)} \quad \text{if $n = 1$, $p \in (2,\infty)$, $r \in (2,\infty)$}
\end{split}
\end{equation}
where $\gamma = \gamma(p,r)$ is a positive constant. This can be summed over all $\ell > 0$, and this gives the desired bound for \eqref{eq:mainest1} when $n = 1$, $p \in (2,\infty)$ and $r \in (2,\infty)$ for $k, \ell$ as in Case 1.

To prove (\ref{eq:localsmooth2}) we use the local smoothing estimate in Theorem~\ref{180320thm1.6}. Suppose $n = 1$, $p = r \in (4,\infty)$. We use (\ref{eq:Lplqlq}) with $q = r = p$. Thus, the left hand side of (\ref{eq:localsmooth2}) is bounded up to a constant by
\begin{equation} \label{eq:0711}
2^{\frac{\ell \alpha}{p}} \| \| \| \chi(u) \mathcal{H}^{(2^{j \alpha} u)}_{\ell-j} P_{j+k} f (x)  \|_{L^p_u} \|_{\ell^p_j} \|_{L^p_x}.
\end{equation}
Consider first
\[
\begin{split}
&\| \| \int_{\R^n}  e^{i x \cdot \xi} \widehat{f}(\xi)  \tilde{m}_u(\xi)  d\xi  \|_{L^p_u}  \|_{L^p_x} \\
&=\| \| \chi(u) \int_{\R^n} e^{i x \cdot \xi} \widehat{f}(\xi) \psi(2^{-k} \xi) \left( e^{i c_{\alpha} u^{-\frac{1}{\alpha-1}} |\xi|^{\alpha'}} a(2^{\ell} \xi, 2^{\ell \alpha} u) + 
e(2^{\ell} \xi, 2^{\ell \alpha} u) \right) d\xi \ \|_{L^p_u}  \|_{L^p_x}.
\end{split}
\]
We first use Fubini's theorem to interchange the integrals in $u$ and $x$, and use H\"ormander--Mikhlin multiplier theorem (for each fixed $u$) to get rid of the multiplier $a(2^{\ell} \xi, 2^{\ell \alpha} u)$. Since $k = \ell(\alpha-1)+O(1)$, this gives
\[
\begin{split}
&\| \| \int_{\R^n}  e^{i x \cdot \xi} \widehat{f}(\xi)  \tilde{m}_u(\xi)  d\xi  \|_{L^p_u}  \|_{L^p_x} \\
\lesssim & 2^{-\ell \alpha / 2} \| \| \chi(u) \int_{\R^n}  e^{i x \cdot \xi} \widehat{f}(\xi) \psi(2^{-k} \xi) e^{i c_{\alpha} u^{-\frac{1}{\alpha-1}} |\xi|^{\alpha'}} d\xi \ \|_{L^p_u}  \|_{L^p_x} + 2^{-\ell N} \|f\|_{L^p(\R)}
\end{split}
\]
for any positive integer $N$.
Thus, Theorem \ref{180320thm1.6} applies, and when $k \geq 0$ we have
\[ 
\begin{split}
&  \| \| \int_{\R^n}  e^{i x \cdot \xi} \widehat{f}(\xi)  \tilde{m}_u(\xi)  d\xi  \|_{L^p_u}  \|_{L^p_x} \lesssim  2^{-\ell \alpha /2} 2^{k \alpha' \left[ \left(\frac{1}{2} - \frac{1}{p} \right) - \frac{1}{p} \right]} \|f\|_{L^p(\R)}  \quad \text{if $n = 1$, $p \in (4,\infty)$}.
\end{split}
\]
But recall that the multiplier of the operator $\chi(u) \mathcal{H}^{(2^{j \alpha} u)}_{\ell-j} P_{j+k}$ is precisely $\tilde{m}_u(2^{-j} \xi)$. By scale invariance, and remembering that $k = \ell(\alpha-1) + O(1)$, we have
\[ 
 \| \| \chi(u) \mathcal{H}^{(2^{j \alpha} u)}_{\ell-j} P_{j+k} f (x)  \|_{L^p_u}  \|_{L^p_x} \lesssim 2^{- 2\ell \alpha/p} \|f\|_{L^p(\R)} \quad \text{if $n = 1$, $p \in (4,\infty)$}.
\]
Replacing $f$ by $\tilde{P}_{j+k} f$, taking $\ell^p_j$ norm on both sides, and using Littlewood--Paley inequality (remember $p \geq 2$), we get 
\[ 
\| \| \| \chi(u) \mathcal{H}^{(2^{j \alpha} u)}_{\ell-j} P_{j+k} f (x)  \|_{L^p_u}  \|_{\ell^p_j} \|_{L^p_x} \lesssim 2^{- 2\ell \alpha/p} \|f\|_{L^p(\R)} \quad \text{if $n = 1$, $p \in (4,\infty)$}.
\]
Thus, (\ref{eq:0711}) is $\lesssim 2^{- \ell \alpha/p} \|f\|_{L^p(\R)}$. This establishes (\ref{eq:localsmooth2}), and our treatment for Case~1 is now complete.

\subsection{Estimates in Cases 2 and 3: Further gains over Case 1} \label{sect4.5}
Next we estimate \eqref{eq:mainest1} for $k, \ell$ as in Case 2. Fix $k, \ell$ such that $k > \ell (\alpha - 1) + C$ for some positive constant~$C$. If $C$ is large enough, then the multiplier for $\mathcal{H}^{(2^{j \alpha} u)}_{\ell-j} P_{j+k} f$ is given by \eqref{eq:multiplier_expansion2}, since the phase function of the oscillatory integral defining the multiplier has no critical point in that case. For $u \in \R$, let $\tilde{m}_u(\xi)$ be the multiplier  
\[
\tilde{m}_u(\xi) = \chi(u) \psi(2^{-k} \xi) e(2^{\ell} \xi, 2^{\ell \alpha} u)
\]
where $e \in S^{-\infty}(\R^{n+1})$ is as in \eqref{eq:multiplier_expansion2}. Then the multiplier of the operator $\chi(u) \mathcal{H}^{(2^{j \alpha} u)}_{\ell-j} P_{j+k}$ is precisely $\tilde{m}_u(2^{-j} \xi)$. For every $N \in \N$ we can write
\[
\tilde{m}_u(\xi) = 2^{-(k+\ell) N} \chi(u) \psi(2^{-k} \xi) \tilde{e}_N(2^{\ell} \xi, 2^{\ell \alpha} u)
\]
for some symbol $\tilde{e}_N \in S^{-\infty}(\R^{n+1})$. Thus, applying Proposition~\ref{prop:Seeger88} as in the proof of \eqref{eq:squarefcn1}, we get
\[ 
\left\| \|V_j^{r} \mathcal{H}^{(u)}_{\ell-j} P_{j+k} f \|_{\ell^r_j} \right\|_p \lesssim_N 2^{-(k+\ell) N} \|f\|_{L^p(\R^n)}
\]
whenever one of the following two conditions is fulfilled: $n = 1$, $p \in [2,\infty)$ and $r \in [2,\infty)$, or $n \geq 2$, $p \in (2 + \frac{4}{n},\infty)$, and $r \in [2,\infty)$. The right hand side in the above display equation can be summed over all $k, \ell$ that satisfies $k > \ell (\alpha -1)+C$ and the standing assumption \eqref{eq:klcaseremain}, and this gives the bound for \eqref{eq:mainest1} for such $p, n, r$ for all $k, \ell$ as in Case 2.\\

Finally we estimate \eqref{eq:mainest1} for $k, \ell$ as in Case 3. Fix $k, \ell$ such that $k < \ell (\alpha - 1) - C$ for some positive constant $C$. As in Case 2, if $C$ is large enough, then the multiplier for $\mathcal{H}^{(2^{j \alpha} u)}_{\ell-j} P_{j+k} f$ is given by \eqref{eq:multiplier_expansion2}. For $u \in \R$, let $\tilde{m}_u(\xi)$ be the multiplier  
\[
\tilde{m}_u(\xi) = \chi(u) \psi(2^{-k} \xi) e(2^{\ell} \xi, 2^{\ell \alpha} u)
\]
where $e \in S^{-\infty}(\R^{n+1})$ is as in \eqref{eq:multiplier_expansion2}. Then the multiplier of the operator $\chi(u) \mathcal{H}^{(2^{j \alpha} u)}_{\ell-j} P_{j+k}$ is precisely $\tilde{m}_u(2^{-j} \xi)$. For every $N \in \N$ we can write
\[
\tilde{m}_u(\xi) = (2^{-\ell \alpha} u)^N \chi(u) \psi(2^{-k} \xi) \tilde{e}_N(2^{\ell} \xi, 2^{\ell \alpha} u)
\]
for some symbol $\tilde{e}_N \in S^{-\infty}(\R^{n+1})$. Thus, applying Proposition~\ref{prop:Seeger88} as in the proof of \eqref{eq:squarefcn1}, we get
\[ 
\left\| \|V_j^{r} \mathcal{H}^{(u)}_{\ell-j} P_{j+k} f \|_{\ell^r_j} \right\|_p \lesssim_N 2^{-\ell \alpha N} \|f\|_{L^p(\R^n)}
\]
whenever one of the following two conditions is fulfilled: $n = 1$, $p \in [2,\infty)$ and $r \in [2,\infty)$, or $n \geq 2$, $p \in (2 + \frac{4}{n}, \infty)$, and $r \in [2,\infty)$. The right hand side in the above display equation can be summed over all $k, \ell$ that satisfies $k < \ell (\alpha -1) - C$ and the standing assumption \eqref{eq:klcaseremain}. This gives the bound for \eqref{eq:mainest1} for such $p, n, r$ for all $k, \ell$ as in Case 3.

We have thus completed the proof of \eqref{eq:mainest_Sect4} for all $p, n, r$ satisfying (\ref{eq:cond_n1}) or (\ref{eq:cond_n2}).

\section{Short jump estimates for $p\le 2$}\label{section:short-jump-small}

In this section, we establish 
\begin{equation} \label{eq:mainest_Sect5}
\left\| \left(\sum_{j\in \Z} |V_j^{r} (\mathcal{H}f)|^{r} \right)^{\frac{1}{r}}\right\|_p \lesssim \|f\|_p,
\end{equation}
whenever $n \geq 2$, $\frac{2n}{2n-1} < p \leq 2$, and $r \in [2, \infty)$. By complex interpolation (see \cite[Theorem 5.1.2]{BL76}) with \eqref{eq:mainest_Sect4}, we will then have \eqref{eq:mainest_Sect5} whenever $n \geq 2$, $p \in (\frac{2n}{2n-1}, \infty)$, and $r \in [2, \infty)$ which concludes the proof of Proposition \ref{main2}.

The key here is the following square function estimate. 
\begin{prop}\label{170713prop5.1}
Let $n\ge 2$, $1<p\le 2$ and $\lambda>1$. Then for any compact time interval $I$, 
\[
\left\|\left(\int_{I} \left| \int_{\R^n}e^{ix\xi}\hat{f}(\xi)e^{it|\xi|^{\lambda}} d\xi\right|^2 dt\right)^{1/2}\right\|_{L^p(\R^n)} \lesssim \|f\|_{W^{\beta,p}(\R^n)},
\]
with $\frac{\beta}{\lambda}=n(\frac{1}{p}-\frac{1}{2})$.
\end{prop} 
The proof of this proposition is postponed to the end of this section.

Now let $\frac{2n}{2n-1} < p \leq 2$, $n \geq 2$, and $r \in [2, \infty)$. We proceed to establish \eqref{eq:mainest_Sect5}. As in Section~\ref{section:short-jump-large}, we decompose $\mathcal{H}^{(u)} f$ as in \eqref{eq:decompl} and \eqref{eq:decompk}, and estimate \eqref{eq:mainest1} for every $k, \ell \in \Z$. The inequalities \eqref{est1} and \eqref{est2} still hold under our current assumptions of $p, n, r$, and these estimates can be summed whenever $\ell \leq -\frac{k}{2(\alpha + 1)}$. Thus, it remains to consider pairs of $(k,\ell)$ for which \eqref{eq:klcaseremain} holds, and we still divide into Cases 1, 2, 3 as before. We will only treat Case 1 here which is the main case; an easy adaptation of this argument gives Cases 2 and 3.

So let $\ell \geq 0$ and $k = \ell(\alpha-1) + O(1)$. By \eqref{eq:Lplqlq} with $q = 2$, the left hand side of \eqref{eq:mainest_Sect5} is bounded by 
\begin{equation} \label{eq:Lpl2l2}
2^{\frac{\ell \alpha}{2}}  \| \| \| \chi(u) \mathcal{H}^{(2^{j \alpha} u)}_{\ell-j} P_{j+k} f (x) \|_{L^2_u} \|_{\ell^2_j} \|_{L^p_x}.
\end{equation}
We analyze the multiplier of $\mathcal{H}^{(2^{j \alpha} u)}_{\ell-j} P_{j+k}$ as before, but in \eqref{eq:square_fcn_Sect4} we use Proposition~\ref{170713prop5.1} instead of Proposition~\ref{LRSsquare} (since now $p \in (\frac{2n}{2n-1},2)$). So instead of \eqref{eq:square_fcn_3}, we get
\begin{equation} \label{eq:square_fcn_4}
\left\|  \left\|  \chi(u) \mathcal{H}^{(2^{j \alpha} u)}_{\ell-j} P_{j+k} f \right\|_{L^2_u} \right\|_{L^p(\R^n)}  \lesssim 
2^{\ell \alpha n \left( \frac{1}{p} - \frac{1}{2} \right)} 2^{-\ell \alpha \frac{n}{2}}  \|f\|_{L^p(\R^n)} = 2^{-\ell \alpha n / p'} \|f\|_{L^p(\R^n)}
\end{equation}
uniformly in $j \in \Z$.
Thus, we apply Proposition~\ref{prop:Seeger88} with $A = 2^{-\ell \alpha n / p'}$ and $B = 2^{\ell N}$ for some large integer $N$ depending only on $\alpha$. This gives 
\[ 
\left\| \left\| \left\|  \chi(u) \mathcal{H}^{(2^{j \alpha} u)}_{\ell-j} P_{j+k} f \right\|_{L^2_u} \right\|_{\ell^2_j} \right\|_{L^p(\R^n)}  \lesssim_{\varepsilon} 
 2^{-\ell \alpha n / p'} 2^{\ell \varepsilon} \|f\|_{L^p(\R^n)}
\]
for all $\varepsilon > 0$. Continuing from \eqref{eq:Lpl2l2}, we see that the left hand side of \eqref{eq:mainest_Sect5} is bounded by 
\[
2^{\ell \alpha/2} 2^{-\ell \alpha n / p'} 2^{\ell \varepsilon} \|f\|_{L^p(\R^n)}.
\]
Since $p \in (\frac{2n}{2n-1},\infty)$, the above exponent of $2$ is negative if $\varepsilon$ is sufficiently small. Thus, we can sum over all $\ell \geq 0$ in this case, establishing the bound for \eqref{eq:mainest1} for all $k, \ell$ in Case 1. A similar argument establishes a bound for \eqref{eq:mainest1} for all $k, \ell$ in Cases 2 and 3. This completes the proof of \eqref{eq:mainest_Sect5}, modulo the proof of Proposition \ref{170713prop5.1}.

\begin{proof}[Proof of Proposition \ref{170713prop5.1}]

We will prove a slightly more general result. Let us write
\[ T_u f(x) = \int_{\R^n} e^{ix\xi} \widehat{f}(\xi) m_u(\xi) d\xi,\]
where
\[m_u(\xi)=e^{iu|\xi|^{\lambda}} (1+|\xi|^2)^{-(\beta+i\gamma)/2}.\]
\begin{thm} \label{thm5.2}
Let $I$ be a compact interval not containing $0$. If $\lambda>1$, $\beta=n\lambda/2$ and $\gamma\in\R$, then 
\[ \Big\| \Big( \int_I |T_u f(x)|^2 du \Big)^{1/2} \Big\|_{L^1(\R^n)} \lesssim \|f\|_{H^1(\R^n)},\]
that is, $T$ maps the Hardy space $H^1(\R^n)$ boundedly into $L_x^1(\R^n; L_u^2(I))$.
\end{thm}
All implied constants may depend on $\lambda,\beta,n,I$, but not on $f,\gamma,x,u$.

Theorem~\ref{thm5.2} implies Proposition \ref{170713prop5.1} via complex interpolation; see, for example, \cite[Chapter IV.6.17]{MR1232192} for a discussion of interpolation between Hardy spaces.
For the scalar theory of the multiplier $m_u$ (for fixed $u$) we refer to \cite{Miy81}, \cite{FS72}, \cite{Fef70}.

Recall that an $H^1$ atom of radius $r$ is a bounded function  $a$ on $\R^n$ that is supported in a ball of radius $r$ and satisfies $\|a\|_\infty\le r^{-n}$ and $\int_{\R^n} a = 0$. The Hardy space $H^1(\R^n)$ is a Banach space consisting of functions of the form $f=\sum_{j} c_j a_j$ with $\sum_{j} |c_j|<\infty$, where the $(a_j)_j$ are $H^1$ atoms. Its norm is defined as
\[ \|f\|_{H^1(\R^n)} = \inf \sum_j |c_j|, \]
where the infimum is taken over all atomic decompositions of $f=\sum_j c_j a_j$.

To prove Theorem~\ref{thm5.2} it suffices to show that
\begin{equation}
\label{eqn:prop5.1.atomicest}
\| T a \|_{L_x^1(L_u^2)} \lesssim 1
\end{equation}
holds for every $H^1$ atom $a$ of radius $r$. We may assume the support of $a$ to be centered at the origin.

For $j>0$, let $P_j$ denote the usual Littlewood--Paley projection with $\widehat{P_j f}=\psi_j \widehat{f}$, $\psi_j$ supported on $|\xi| \simeq 2^j$. Let $\widehat{P_0 f}=\psi_{0} \widehat{f}$ where $\psi_0$ is such that
\[ 1 = \psi_{0} + \sum_{j>0} \psi_j. \] (Note that $P_0$ here is actually $P_{\leq 0}$ from the previous section.)
For $j\ge 0$ we denote by $\widetilde{\psi}_j$ a smooth positive function that equals one on the support of $\psi_j$ and whose support is contained in a small neighborhood of the support of $\psi_j$.
Define 
\[ K_u^{(j)} (x) =  \widehat{m_u}*\widehat{\psi_j} (x) = \int_{\R^n} e^{ix\xi+iu|\xi|^{\lambda}} (1+|\xi|^2)^{-\beta/2} \psi_j(\xi) d\xi. \]

Before we begin, we record the following pointwise estimates for $K_u^{(j)}(x)$.
From estimating the second derivative of the phase we obtain
\begin{equation}\label{eqn:prop5.1.kernelest1}
|K_u^{(j)}(x)| \lesssim 2^{-jn(\lambda-1)}\;\;\text{for all}\;x\in\R^n, u\in I.
\end{equation}
Here we used that $\beta = n\lambda/2$.
Estimating the first derivative of the phase we obtain
\begin{equation}\label{eqn:prop5.1.kernelest2}
 |K_u^{(j)}(x)| \lesssim_N 2^{-j\beta} 2^{jn} (2^j |x|)^{-N}\;\;\text{for}\;|x|\gtrsim 2^{j(\lambda-1)}, u\in I.
\end{equation}
for all $N\ge 0$. Note that these estimates are uniform in $u\in I$.

Let us prove the main estimate \eqref{eqn:prop5.1.atomicest}. The first step is to apply the triangle inequality:
\[ \|T a\|_{L^1_x(L^2_u)} \le \sum_{j\ge 0} \| K_u^{(j)} * a\|_{L^1_x(L^2_u)}. \]
We will estimate the summand in two different ways: in particular, it will be shown below that
\begin{equation}\label{eqn:prop5.1.est1}
 \| K_u^{(j)} * a\|_{L^1_x(L^2_u)} \lesssim (2^j r)^{-n/2} + 2^{-j\beta}
\end{equation}
\begin{equation}\label{eqn:prop5.1.est2}
 \| K_u^{(j)} * a\|_{L^1_x(L^2_u)} \lesssim 2^j r.
\end{equation}
These estimates immediately imply \eqref{eqn:prop5.1.atomicest}.

To prove \eqref{eqn:prop5.1.est1} we first split up\footnote{$C$ is a constant that may depend on the parameters $\lambda,\beta,n$.}
 the integral in $x$:
\[ \| K_u^{(j)} * a\|_{L^1_x(L^2_u)}\le \mathbf{I} + \mathbf{II},\,\text{where}\]
\[ \mathbf{I} = \| K_u^{(j)} * a\|_{L^1_x(\R^n\backslash B(C2^{j(\lambda-1)}+r);L^2_u(I))} ,\,\text{and}\]
\[ \mathbf{II} = \| K_u^{(j)} * a\|_{L^1_x(B(C2^{j(\lambda-1)}+r);L^2_u(I))}. \]
We claim that $\mathbf{I}\lesssim_N 2^{-jN}$. Indeed, we see from \eqref{eqn:prop5.1.kernelest2} that
\begin{align*}
\mathbf{I} &\le \int_{|x|\ge C2^{j(\lambda-1)}+r} \Big(\int_I \Big( \int_{|y|\le r} |K_u^{(j)}(x-y) a(y)| dy \Big)^2 du \Big)^{1/2} dx\\
& \lesssim_N 2^{-j\beta} 2^{jn} 2^{-jN} \int_{|x|\ge C2^{j(\lambda-1)}+r} \int_{|y|\le r} |x-y|^{-N} |a(y)| dy dx \\
& \le 2^{-j\beta} 2^{jn} 2^{-jN} \|a\|_1 \int_{|x|\gtrsim 2^{j(\lambda-1)}} |x|^{-N} dx \lesssim 2^{-\beta j} 2^{jn\lambda} 2^{-jN\lambda},
\end{align*}
which implies the claim (since $N$ is arbitrary). To estimate the second part we use the Cauchy-Schwarz inequality:
\[ \mathbf{II} \le (C2^{j(\lambda-1)}+r)^{n/2} \|K_u^{(j)}*a\|_{L^2_x(L^2_u)}.\]
Then we have by the Fubini and Plancherel theorems that
\[\|K_u^{(j)}*a\|_{L^2_x(L^2_u)} = \Big\| \|K_u^{(j)}*a\|_2 \Big\|_{L^2_u(I)}\lesssim 2^{-j\beta} \|a\|_2 \lesssim 2^{-j\beta} r^{-n/2},\]
which implies that
\[ \mathbf{II} \lesssim (2^{j(\lambda-1)n/2} + r^{n/2}) 2^{-j\beta} r^{-n/2}=(2^j r)^{-n/2} + 2^{-j\beta},\]
as desired (we used that $\beta=n\lambda/2$). This proves \eqref{eqn:prop5.1.est1}.\\

It remains to show \eqref{eqn:prop5.1.est2}. Clearly we have 
\begin{equation}\label{eqn:prop5.1.maxesta} 
\| K_u^{(j)} * a\|_{L^1_x(L^2_u)} \lesssim \|\sup_{u\in I} |K_u^{(j)}*a| \|_1.
\end{equation}
We claim that
\begin{equation}\label{eqn:prop5.1.maxest}
\|\sup_{u\in I} |K_u^{(j)}*a| \|_1 \lesssim \|P_j a\|_1.
\end{equation}
To see this replace $\psi_j$ by $\widetilde{\psi}_j$ in the definition of $K_u^{(j)}$ and call the resulting kernel $\widetilde{K_u^{(j)}}$. Then we have
\[ K_u^{(j)}*a = \widetilde{K_u^{(j)}}*P_j a.\]
It is clear that $\widetilde{K_u^{(j)}}$ satisfies the same pointwise estimates \eqref{eqn:prop5.1.kernelest1}, \eqref{eqn:prop5.1.kernelest2} (possibly with larger constants). Thus, there exists a positive function $w_j$ on $\R^n$ such that $\|w_j\|_1\lesssim 1$ and
\[ |\widetilde{K_u^{(j)}}(x)| \le w_j(x) \]
for all $x\in\R^n$ and $u\in I$.
As a consequence,
\[ \|\sup_{u\in I} |K_u^{(j)}*a| \|_1\le \|\sup_{u\in I} |\widetilde{K_u^{(j)}}|*|P_j a| \|_1 \le \|w_j * |P_j a|\|_1\lesssim \|P_j a\|_1,  \]
which is our claim \eqref{eqn:prop5.1.maxest}.
But by the mean zero property of $a$ and the mean value theorem we have
\[ P_j a (x) = \int_{\R^n} (\widehat{\psi_j}(x-y)-\widehat{\psi_j}(x)) a(y) dy = -\int_{\R^n} \int_0^1 y\cdot \nabla \widehat{\psi_j} (x-ty) dt \, a(y) dy.\]
This implies that
\[ \|P_j a\|_1 \le \int_{\R^n} \int_{|y|\le r} \int_0^1 |y| |\nabla \widehat{\psi}_j(x-ty)| |a(y)| dt dy dx\lesssim 2^j r.\]
In view of (\ref{eqn:prop5.1.maxesta}) and (\ref{eqn:prop5.1.maxest}), this establishes \eqref{eqn:prop5.1.est2}.

\end{proof}

\section{A Counterexample: The proof of Theorem \ref{170712thm1.2}}\label{section:counter}

Let $\phi$ be a smooth test function supported in the annulus $1/2\le |\xi|\le 2$ and define $f_k$ for $k\in\Z$ by \[\widehat{f_k}(\xi) = \phi(2^{-k} \xi).\]
On the one hand, clearly,
\[ \|f_k\|_p = 2^{nk} \left(\int_{\mathbb{R}^n} |\widehat{\phi}(2^k x)|^p dx\right)^{1/p}\approx 2^{nk/p'}.\]
On the other hand, we claim that
\begin{equation}\label{eqn:vrnormclaimgen}
\|V^r \{ \mathcal{H}^{(u)} f_k \colon u \in \mathbb{R}\} \|_p \gtrsim 2^{k(-n(\alpha-1)\frac1{p'} +\frac\alpha r)}.
\end{equation}
If the \eqref{170614e1.2} were to hold, then plugging in $f_k$ into the estimate \eqref{170614e1.2} and letting $k\to\infty$, we see that
\[ -n(\alpha-1)\frac1{p'} +\frac\alpha r\le \frac{n}{p'},
\]
which is equivalent to that $p'\le nr$.

For simplicity we will verify this only in the case $n=1$, $K(t)=\textrm{p.v.} \frac{1}{t}$, $\alpha=2$. The general case can be treated in the same manner.
In this case \eqref{eqn:vrnormclaimgen} takes the form
\begin{equation}\label{eqn:vrnormclaim}
\|V^r \{ \mathcal{H}^{(u)} f_k \colon u \in \mathbb{R}\}\|_p \gtrsim 2^{k(-1+\frac1p+\frac2r)}.
\end{equation}
We can choose $\varphi$ such that
\[ \frac1t = \sum_{j\in\Z} \varphi_j(t) \]
for all $t\not=0$, where $\varphi_j(t)=2^{-j} \varphi(2^{-j} t)$.
By Fourier inversion we have
\[ \mathcal{H}^{(u)} f_k (x) = \sum_{j\in\Z} \iint \phi(2^{-k}\xi) e^{ix\xi} e^{-it\xi+iut^2} \varphi_{j+k}(t) dt d\xi \]
\[ = \sum_{j\in\Z} 2^k \int e^{i2^k x\xi} \phi(\xi) \int e^{-i2^{2k}t\xi+i2^{2k} ut^2} \varphi_j(t) dt d\xi = \sum_{j\in\Z} \mathbf{I}_j.\]
(Keep in mind that $\mathbf{I}_j$ also depends on $k,x,u$.)
Let us take $u\in[1,2]$. Then the phase of the oscillatory integral in $t$ has no critical points if $|j|>10$. This motivates us to set
\[ \mathbf{I}^{\textrm{main}} = \sum_{|j|\le 10} \mathbf{I}_j,\,\text{and}\,\,\mathbf{I}^{\textrm{err}} = \sum_{|j|>10} \mathbf{I}_j.\]
Write $B=[2^k,2^{k+1}]$ and estimate
\begin{align*} 
\|V^r \{ \mathcal{H}^{(u)} f_k \colon u \in \mathbb{R}\}\|_p 
&\ge \|V^r \{ \mathcal{H}^{(u)} f_k \colon u \in [1,2]\}\|_{L^p(B)} \\
&\ge \|V^r \{ \mathbf{I}^{\textrm{main}} \colon u \in [1,2]\}\|_{L^p(B)} - \|V^r \{ \mathbf{I}^{\textrm{err}} \colon u \in [1,2]\}\|_{L^p(B)}.
\end{align*}
In order to verify \eqref{eqn:vrnormclaim} it suffices to show that
\begin{equation}\label{eqn:maintermclaim}
\|V^r \{ \mathbf{I}^{\textrm{main}} \colon u \in [1,2]\}\|_{L^p([2^k,2^{k+1}])} \gtrsim 2^{k (-1+\frac1p+\frac2r)}
\end{equation}
and
\begin{equation}\label{eqn:errtermclaim}
\|V^r \{ \mathbf{I}^{\textrm{err}} \colon u \in [1,2]\}\|_{L^p([2^k,2^{k+1}])} \lesssim 2^{-2k}.
\end{equation}

We begin with the proof of \eqref{eqn:maintermclaim}. Write
\[ \mathbf{I}^{\text{main}} = 2^k \int e^{i2^k x\xi} \phi(\xi) \int e^{-i2^{2k}t\xi+i2^{2k} ut^2} \rho(t) dt d\xi,\]
where $\rho(t) = \sum_{|j|\le 10} \varphi_j(t)$. Note that the phase of the integral in $t$ has a critical point at $t_c = \frac{\xi}{2u}$. By the principle of stationary phase (\cite[Chapter VIII.5.7]{MR1232192} or \cite[Theorem 1.2.1]{MR1205579}) we have\[ \int e^{-i2^{2k}t\xi+i2^{2k} ut^2} \rho(t) dt = 2^{-k} c_0 e^{i2^{2k} c_1 \xi^2 u^{-1}} u^{-\frac12} \rho(\xi/(2u)) + O(2^{-2k}).\]
Here $c_0,c_1$ are non-zero constants. To simplify the calculation, let us take $c_0=c_1=1$. Thus, the main contribution to $\mathbf{I}^{\textrm{main}}$ is
\[ \int e^{i 2^{2k}(\tilde{x}\xi+\xi^2 u^{-1})} a(\xi,u) d\xi,\]
where $x=2^k \tilde{x}\in [2^k,2^{k+1}]$ and $a(\xi,u) =  \phi(\xi) u^{-\frac12} \rho(\xi/(2u)).$
Note that the $u$-derivative of the error term coming from stationary phase is also $O(2^{-2k})$. Therefore that term contributes only $O(2^{-2k})$ to the variation-norm and we can ignore it. From another application of the stationary phase principle we see that the previous integral can essentially be written in the form
\[ 2^{-k} e^{ix^2 u} b(u) + O(2^{-2k}),\]
where $b(u)=\phi(\tilde{x}u/2) \rho(\tilde{x}/4)$. Let $u_\ell = \ell \pi / x^2$ for $x^2/\pi <\ell<2x^2/\pi$.
Then 
\[ \left(\sum_{x^2/\pi <\ell<2x^2/\pi} |e^{ix^2 u_{\ell+1}}-e^{ix^2 u_\ell}|^r\right)^{1/r} \approx 2^{\frac{2k}{r}} \]
which implies the claim \eqref{eqn:maintermclaim} (the contribution of $b(u)$ is negligible).

It remains to treat $\mathbf{I}^{\textrm{err}}$. Compute
\[\partial_u \mathbf{I}_j= i 2^{3k+2j} \int \phi(\xi) e^{i2^k x\xi} \int e^{-i2^{2k+j}t\xi+i2^{2k+2j}ut^2} t^2 \varphi(t) dt d\xi\]
\[= i 2^{3k+2j} \int \widehat{\phi}(2^{2k+j}t-2^k x) e^{i2^{2k+2j}ut^2} t^2 \varphi(t) dt.\]
Observe that if $x\in[2^k,2^{k+1}]$, $|t|\in [1/2,2]$ and $|j|>10$ we have
\[ |2^{2k+j}t-2^k x|\approx 2^{2k} \max(1,2^j). \]
Since $\widehat{\phi}$ decays rapidly, we obtain
\[ V^r \{ \mathbf{I}_j \colon u \in [1,2]\} \lesssim \|\partial_u \mathbf{I}_j\|_{L^1_u([1,2])} \lesssim_N 2^{3k+2j-2Nk} \min(1,2^{-Nj}) \]
for every $N\ge 1$. Taking $N$ large enough ($N=3$ suffices) and summing over $|j| > 10$, we obtain \eqref{eqn:errtermclaim}.

\section{Proof of Proposition~\ref{prop:Seeger88}}\label{180320section7}

In this section we provide a proof of Proposition~\ref{prop:Seeger88}, which was stated in Jones, Seeger and Wright \cite{MR2434308} without proof. Indeed, Proposition~\ref{prop:Seeger88} is a vector-valued analog of Theorem 1 of Seeger \cite{MR955772}. The proof of Proposition~\ref{prop:Seeger88} follows closely that of the scalar-valued case in Seeger \cite{MR955772}. On the other hand, at one point in the scalar-valued case, Seeger used a duality argument between $L^p$ and $L^{p'}$, which is not available in the vector-valued setting. This is why we had to assume that hypothesis (\ref{eq:assump_1piece}) holds not just for $s = p$, but also at $s = 2$.

To prove Proposition~\ref{prop:Seeger88}, one key tool is the Fefferman-Stein sharp function. Let $B$ be a Banach space. For us we will only need the case $B = \ell^2(\Z) L^2(I) $. For each measurable function $F \colon \R^n \to B$, define its Hardy-Littlewood maximal function $\mathcal{M} F$ by
\[
\mathcal{M} F(x) = \sup_{x \in Q} \fint_Q |F(y)|_B dy
\]
for each $x \in \R^n$, where the supremum is over all cubes $Q$ containing $x$. Also define the sharp function $F^{\sharp}$ of $F$ by
\[
F^{\sharp}(x) = \sup_{x \in Q} \fint_{Q} |F(y) - F_Q|_B dy
\]
where $F_Q = \fint_Q F(y) dy$; again the supremum is over all cubes $Q$ containing $x$. We have the following lemma about $F^{\sharp}$:

\begin{lem} \label{lem:FSsharp}
Suppose $0 < p < \infty$. Let $F \in L^{p_0}(\R^n,B)$ for some $0 < p_0 \leq p$. If $F^{\sharp} \in L^p(\R^n)$, then $\mathcal{M} F \in L^p(\R^n)$, and 
\[
\|\mathcal{M} F \|_{L^p(\R^n)} \lesssim_{n,p} \|F^{\sharp}\|_{L^p(\R^n)}.
\]
\end{lem}

We give a proof of this lemma at the end.

Now given $f \in L^p \cap L^2(\R^n)$, define $Tf \colon \R^n \to B$ by
\[
Tf(x) = \left( T_{u,j} f(x) \right)_{u \in I, j \in \Z}.
\]
Note $Tf \in L^2(\R^n,B)$. Then
\[
\left\| \| \|T_{u,j} f\|_{L^2(I)} \|_{\ell^2(\Z)} \right\|_{L^p(\R^n)}
= \left\| |Tf|_B \right\|_{L^p(\R^n)}
\leq \left\| \mathcal{M}(Tf) \right\|_{L^p(\R^n)}
\lesssim_{n,p} \|(Tf)^{\sharp}\|_{L^p(\R^n)},
\]
where in the last inequality we invoked the lemma with $p_0 = 2$.
Note that for a.e. $x \in \R^n$,
\[
(Tf)^{\sharp}(x) \simeq \fint_{Q_x} |Tf(y) - (Tf)_{Q_x}|_B dy
\]
for some cube $Q_x$ containing $x$; we may choose $Q_x$ such that the side length of $Q_x$ is $2^{r(x)}$ for some integer $r(x)$. Then we split
\[
(Tf)^{\sharp}(x) \lesssim \sigma_1((T_{u,j}f),x) + \sigma_2((P_j f),x)
\]
where $N$ is a positive integer to be chosen; here
\[
\sigma_1(G,x) = \fint_{Q_x} \left( \sum_{|j +r(x)| \leq N} \|G_{u,j}(y)-(G_{u,j})_{Q_x} \|_{L^2(I)}^2 \right)^{1/2} dy
\]
\[
\sigma_2(H,x) = \fint_{Q_x} \left( \sum_{|j +r(x)| > N} \|T_{u,j} H_j (y) - (T_{u,j} H_j)_{Q_x} \|_{L^2(I)}^2 \right)^{1/2} dy
\]
for any functions $G = (G_{u,j}) \colon \R^n \to B$ and $H = (H_j) \colon \R^n \to \ell^2(\Z)$. We claim that
\begin{equation} \label{eq:sigma1}
\|\sigma_1(G,x)\|_{L^p(\R^n)} \lesssim N^{\frac{1}{2}-\frac{1}{p}} \| \| \| G_{u,j}(x) \|_{L^2(I)} \|_{\ell^p(\Z)} \|_{L^p(\R^n)}
\end{equation}
\begin{equation} \label{eq:sigma2}
\|\sigma_2(H,x)\|_{L^p(\R^n)} \lesssim (A+B 2^{-N}) \| \| H_j(x)\|_{\ell^2(\Z)} \|_{L^p(\R^n)}
\end{equation}
for any $G = (G_{u,j}) \colon \R^n \to B$ and $H = (H_j) \colon \R^n \to \ell^2(\Z)$.
But when $G_{u,j} = T_{u,j} f$, we have
\begin{align*}
\| \| \| G_{u,j}(x) \|_{L^2(I)} \|_{\ell^p(\Z)} \|_{L^p(\R^n)} 
&= \| \| \| T_{u,j}f(x) \|_{L^2(I)} \|_{L^p(\R^n)}  \|_{\ell^p(\Z)}\\
&\lesssim A \| \| P_j f(x) \|_{L^p(\R^n)} \|_{\ell^p(\Z)} \\
&\lesssim A \| \| P_j f(x) \|_{\ell^2(\Z)} \|_{L^p(\R^n)} \\
&\lesssim A \|f\|_{L^p(\R^n)} 
\end{align*}
(we used assumption (\ref{eq:assump_1piece}) in the second inequality, $p \in [2, \infty)$ in the third inequality, and Littlewood--Paley inequality in the last). Also, when $H_j = P_j f$, we have
\[ \| \| H_j(x)\|_{\ell^2(\Z)} \|_{L^p(\R^n)} \lesssim \|f\|_{L^p(\R^n)}. \]
Hence
\[
\|(Tf)^{\sharp}\|_{L^p(\R^n)} \lesssim (AN^{\frac{1}{2}-\frac{1}{p}} + A + B2^{-N} ) \|f\|_{L^p(\R^n)}.
\]
Choosing $N \simeq \log(2+\frac{B}{A})$ gives the desired conclusion of the proposition. It remains to prove (\ref{eq:sigma1}) and (\ref{eq:sigma2}).

To prove (\ref{eq:sigma1}), we interpolate between $p=2$ and $p=\infty$. Indeed, we prove
\begin{equation} \label{eq:sigma1a}
\|\sigma_1(G,x)\|_{L^2(\R^n)} \lesssim \| \| \| G_{u,j}(x) \|_{L^2(I)} \|_{\ell^2(\Z)} \|_{L^2(\R^n)}
\end{equation}
\begin{equation} \label{eq:sigma1b}
\|\sigma_1(G,x)\|_{L^{\infty}(\R^n)} \lesssim N^{\frac{1}{2}} \| \| \| G_{u,j}(x) \|_{L^2(I)} \|_{\ell^{\infty}(\Z)} \|_{L^{\infty}(\R^n)}
\end{equation}
The desired estimate (\ref{eq:sigma1}) then follows by complex interpolation and linearizing $\sigma_1$.

The estimate (\ref{eq:sigma1a}) follows since
\begin{equation} \label{eq:sigma1_ptwise}
\sigma_1(G,x) \leq 2 \fint_{Q_x} \left( \sum_{|j +r(x)| \leq N} \|G_{u,j}(y) \|_{L^2(I)}^2 \right)^{1/2} dy
\end{equation}
so
\[
\sigma_1(G,x) \lesssim \fint_{Q_x} \| \|G_{u,j}(y) \|_{L^2(I)} \|_{\ell^2(\Z)} dy  \lesssim M \| \|G_{u,j} \|_{L^2(I)} \|_{\ell^2(\Z)} (x)
\]
where $M$ is the standard (scalar-valued) Hardy-Littlewood maximal function on $\R^n$. Hence 
\[
\|\sigma_1(G,x) \|_{L^2(\R^n)} \lesssim \| \| \|G_{u,j}(x) \|_{L^2(I)} \|_{\ell^2(\Z)} \|_{L^2(\R^n)}
\]
as in (\ref{eq:sigma1a}).

To prove (\ref{eq:sigma1b}), note that for each $x \in \R^n$, we have, from (\ref{eq:sigma1_ptwise}), that
\[
\sigma_1(G,x) \leq 2 \sup_{y \in \R^n} \left( \sum_{|j +r(x)| \leq N} \|G_{u,j}(y) \|_{L^2(I)}^2 \right)^{1/2} 
\lesssim N^{1/2} \sup_{y \in \R^n} \sup_{j \in \Z} \|G_{u,j}(y)\|_{L^2(I)},
\]
with constants uniform in $x$. This gives (\ref{eq:sigma1b}).

Next, to prove (\ref{eq:sigma2}), we will prove
\begin{equation} \label{eq:sigma2a}
\|\sigma_2(H,x)\|_{L^2(\R^n)} \lesssim A \| \|H_j(x)\|_{\ell^2(\Z)} \|_{L^2(\R^n)}
\end{equation} 
\begin{equation} \label{eq:sigma2b}
\|\sigma_2(H,x)\|_{L^{\infty}(\R^n)} \lesssim (A+B 2^{-N}) \| \|H_j(x)\|_{\ell^2(\Z)} \|_{L^{\infty}(\R^n)}
\end{equation} 
The desired estimate (\ref{eq:sigma2}) then follows by complex interpolation and linearizing $\sigma_2$.

To prove (\ref{eq:sigma2a}), note that
\[
\sigma_2(H,x) \leq 2 \fint_{Q_x} \left( \sum_{|j+r(x)|>N} \|T_{u,j} H_j\|_{L^2(I)}^2 \right)^{1/2} dy
\]
so
\[
\sigma_2(H,x)
\lesssim  \fint_{Q_x} \| \|T_{u,j} H_j\|_{L^2(I)} \|_{\ell^2(\Z)} dy
\lesssim M \| \| T_{u,j} H_j \|_{L^2(I)} \|_{\ell^2(\Z)} (x)
\]
Hence
\[
\| \sigma_2(H,x) \|_{L^2(\R^n)} \lesssim \| \| \| T_{u,j} H_j(x) \|_{L^2(I)} \|_{\ell^2(\Z)} \|_{L^2(\R^n)}.
\]
We commute the $\ell^2(\Z)$ norm outside. Since 
\begin{equation} \label{eq:problematic}
\| \| T_{u,j} H_j(x) \|_{L^2(I)} \|_{L^2(\R^n)}
\lesssim A \|H_j(x) \|_{L^2(\R^n)}
\end{equation}
one can conclude that
\[
\| \sigma_2(H,x) \|_{L^2(\R^n)}
\lesssim A \| \|H_j(x) \|_{L^2(\R^n)} \|_{\ell^2(\Z)}
\]
which gives (\ref{eq:sigma2a}) upon a further change in the order of the norms on the right hand side.

Now we proceed to prove (\ref{eq:sigma2b}). For each $x \in \R^n$, we decompose
\[
H_j(y) = (\chi_{2Q_x} H_j)(y) + (\chi_{(2Q_x)^c} H_j)(y) 
\]
for all $y \in \R^n$. Then we plug this back into the formula for $\sigma_2(H,x)$. We find that
\[
\sigma_2(H,x) \lesssim \mathbf{I}(x) + \mathbf{II}(x),
\]
where
\[
\mathbf{I}(x) = \fint_{Q_x} \left( \sum_{j \in \Z} \|T_{u,j} (\chi_{2Q_x} H_j)(y)- [T_{u,j} (\chi_{2Q_x} H_j)]_{Q_j}\|_{L^2(I)}^2 \right)^{1/2} dy
\]
\[
\mathbf{II}(x) = \fint_{Q_x} \left( \sum_{|j+r(x)| > N} \left\|T_{u,j} (\chi_{(2Q_x)^c} H_j)(y) - [T_{u,j} (\chi_{(2Q_x)^c} H_j)]_{Q_j} \right\|_{L^2(I)}^2 \right)^{1/2} dy.
\]
We estimate $\mathbf{I}(x)$ by
\begin{align*}
\mathbf{I}(x) 
&\leq 2 \fint_{Q_x} \left( \sum_{j \in \Z} \|T_{u,j} (\chi_{2Q_x} H_j)(y) \|_{L^2(I)}^2 \right)^{1/2} dy \\
& \lesssim  \left(  \fint_{Q_x} \sum_{j \in \Z} \|T_{u,j} (\chi_{2Q_x} H_j)(y) \|_{L^2(I)}^2  dy  \right)^{1/2} \\
& \lesssim \frac{1}{|Q_x|^{1/2}} \| \| \| T_{u,j} (\chi_{2Q_x} H_j)(y) \|_{L^2(I)} \|_{L^2(\R^n)} \|_{\ell^2(\Z)} \\
& \lesssim \frac{A}{|Q_x|^{1/2}} \| \| (\chi_{2Q_x} H_j)(y) \|_{L^2(\R^n)} \|_{\ell^2(\Z)}
\end{align*}
where in the last inequality we used the estimate (\ref{eq:problematic}). Then
\begin{align*}
\mathbf{I}(x) \lesssim 
A \left( \fint_{2Q_x} \| H_j(y) \|_{\ell^2(\Z)} dy \right)^{1/2} \lesssim A \sup_{y \in \R^n} \| H_j(y) \|_{\ell^2(\Z)},
\end{align*}
which shows that
\[
\|\mathbf{I}(x)\|_{L^{\infty}(\R^n)} \lesssim A \| \|H_j(x) \|_{\ell^2(\Z)} \|_{L^{\infty}(\R^n)}.
\]
Next we estimate $\mathbf{II}(x)$. Let $K_{u,j}$ be the convolution kernel of $T_{u,j}$. Then 
\[
K_{u,j}(x) = 2^{jn} K_{u}(2^j x)
\]
where
\[
K_u(x) = \int_{\R^n} \tilde{m}_u(\xi) e^{2\pi i x \cdot \xi} d\xi
\]
Now by our assumption on $\partial_{\xi}^{\tau} \tilde{m}_u(\xi)$, we have
\[
\sup_{u \in I} |K_u(x)| + \sup_{u \in I} |\nabla_x K_u(x)| \lesssim \frac{B}{(1+|x|)^{n+1}}.
\]
We claim now
\begin{equation} \label{eq:CZest}
\sup_{y, z \in Q_x} \left( \sum_{|j+r(x)| > N} \left( \int_{(2Q_x)^c} \sup_{u \in I} | K_{u,j}(y-w) - K_{u,j}(z-w) | dw \right)^2 \right)^{1/2} 
\lesssim B2^{-N}
\end{equation}
uniformly for $x \in \R^n$. Indeed, suppose $y, z \in Q_x$, and $j + r(x) > -N$. Then 
\begin{align*}
\int_{(2Q_x)^c} \sup_{u \in I} | K_{u,j}(y-w) - K_{u,j}(z-w) | dw
\end{align*}
can be estimated as
\begin{align*}
&\lesssim \int_{(2Q_x)^c} \sup_{u \in I} (| K_{u,j}(y-w) | + | K_{u,j}(z-w) | ) dw \\
&\lesssim 2^{jn} \int_{|w-x| \gtrsim 2^{r(x)}} \frac{B}{(2^j |w-x|)^{n+1}}  dw\lesssim B 2^{-j-r(x)}.
\end{align*}
On the other hand, if $y, z \in Q_x$, and $j + r(x) < -N$, then 
\begin{align*}
\int_{(2Q_x)^c} \sup_{u \in I} | K_{u,j}(y-w) - K_{u,j}(z-w) | dw
\end{align*}
is bounded by a constant times
\begin{align*}
&\int_0^1 \int_{(2Q_x)^c} \sup_{u \in I} |y-z| |\nabla_x K_{u,j} ((1-t)y+tz-w)| dw dt \\
&\lesssim 2^j 2^{r(x)} \int_{\R^n} 2^{jn} |(\nabla_x K_u)(2^j w)|  dw \lesssim B 2^{j+r(x)}.
\end{align*}
Summing over $j$ such that $j + r(x) > N$ and $j + r(x) < -N$ respectively, we see that (\ref{eq:CZest}) follows.

Finally, it suffices to observe that
\begin{align*}
\mathbf{II}(x) & \lesssim \fint_{Q_x} \fint_{Q_x} \left( \sum_{|j+r(x)| > N} \left\|T_{u,j} (\chi_{(2Q_x)^c} H_j)(y) - T_{u,j} (\chi_{(2Q_x)^c} H_j)(z) \right\|_{L^2(I)}^2 \right)^{1/2} dy dz
\end{align*}
\begin{align*}
& = \sup_{y, z \in Q_x} \left( \sum_{|j+r(x)| > N} \left\| \int_{(2Q_x)^c} [K_{u,j}(y-w) - K_{u,j}(z-w)] H_j(w) dw \right\|_{L^2(I)}^2 \right)^{1/2} \\
& \lesssim \sup_{y, z \in Q_x} \left( \sum_{|j+r(x)| > N} \left( \int_{(2Q_x)^c} \left\| K_{u,j}(y-w) - K_{u,j}(z-w) \right\|_{L^2(I)} dw \right)^2 \right)^{1/2} \|\|H_j\|_{\ell^{\infty}(\Z)} \|_{L^{\infty}(\R^n)} \\
& \lesssim \sup_{y, z \in Q_x} \left( \sum_{|j+r(x)| > N} \left( \int_{(2Q_x)^c} \sup_{u \in I} | K_{u,j}(y-w) - K_{u,j}(z-w) | dw \right)^2 \right)^{1/2} \|\|H_j\|_{\ell^2(\Z)} \|_{L^{\infty}(\R^n)}.
\end{align*}
Invoking (\ref{eq:CZest}) yields
\[
\|\mathbf{II}(x)\|_{L^{\infty}(\R^n)} \lesssim B 2^{-N} \| \|H_j(x)\|_{\ell^2(\Z)} \|_{L^{\infty}(\R^n)},
\]
which together with our earlier estimate about $\|\mathbf{I}(x)\|_{L^{\infty}(\R^n)}$ gives (\ref{eq:sigma2b}).

\begin{proof}[Proof of Lemma~\ref{lem:FSsharp}]
The key is a relative distribution inequality. Fix the Banach space $B$. For any $n \geq 1$, we claim that there exists $b_n \in (0,1)$, such that for any $b, c > 0$ with $b \leq b_n$, we have
\[
|\{x \in \R^n \colon \mathcal{M} F(x) > \alpha, F^{\sharp}(x) \leq c \alpha\}| \lesssim_n c |\{x \in \R^n \colon \mathcal{M} F(x) > b \alpha\}|
\]
If this is true, then by taking $c$ sufficiently small, we can use Lemma 2 of Chapter IV.3.5 of \cite{MR1232192} (see also the Remark on the bottom of p.152 there) and conclude the proof of Lemma~\ref{lem:FSsharp}.

To prove the above relative distributional inequality, let $b \in (0,1)$ first. Let $F \in L^p(\R^n,B)$. Decompose the open set $\{x \in \R^n \colon \mathcal{M} F(x) > b \alpha\}$ into an essentially disjoint union of Whitney cubes $\{Q\}$, so that the distance of each $Q$ from the complement of this set is bounded by $4$ times the diameter of $Q$. Now since $\{x \in \R^n \colon \mathcal{M} F(x) > \alpha, F^{\sharp}(x) \leq c \alpha\}$ is a subset of $\{x \in \R^n \colon \mathcal{M} F(x) > b \alpha\}$, we just need to show that for each Whitney cube $Q$ as above, we have
\[
|\{x \in Q \colon \mathcal{M} F(x) > \alpha, F^{\sharp}(x) \leq c \alpha\}| \lesssim_n c |Q|.
\]
This inequality would be trivial if the set on the left hand side were empty. So let's assume there exists a point $x_0 \in Q$ such that $F^{\sharp}(x_0) \leq c\alpha$.  Now let $\tilde{Q}$ be any cube that intersects $Q$ and that has diameter at least that of $Q$. Then $20\tilde{Q}$ will contain a point $y$ where $\mathcal{M} F(y) \leq b \alpha$. Hence $\fint_{\tilde{Q}} |F|_B \leq 20^n b\alpha$ for all such cubes $\tilde{Q}$. If $x \in Q$ and $\mathcal{M} F(x) > \alpha$, then by taking $b < 20^{-n}$, we see that $\mathcal{M} (F \chi_{3Q})(x) > \alpha$. We also have $\fint_{3Q} |F|_B \leq 20^n b\alpha$. Thus,
\[
\{x \in Q \colon \mathcal{M} F(x) > \alpha, F^{\sharp}(x) \leq c \alpha \} \subset\]
\[\hspace{4cm} \left \{x \in Q \colon \mathcal{M} \left(F \chi_{3Q} - \fint_{3Q} F \right)(x) > (1-20^n b)\alpha \right \},
\]
whose measure is bounded by 
\[
\frac{C_n}{(1-20^n b) \alpha} \int_Q |F \chi_{3Q}(y) - F_{3Q}|_B dy 
\leq \frac{C_n}{(1-20^n b) \alpha} |3Q| F^{\sharp}(x_0) \leq \frac{3^n C_n}{(1-20^n b)} c|Q|
\]
where $C_n$ is the constant arising in the weak-type (1,1) bound of $\mathcal{M} \colon L^1(\R^n,B) \to L^{1,\infty}(\R^n)$.
This proves the desired relative distributional inequality.
\end{proof}

\section{Appendix: An improved local smoothing estimate}\label{180320section8}

In this section we prove Theorem \ref{180320thm1.6}.
Let $\chi: \R^n\to \R$ be a non-negative smooth bump function supported on $1\le |\xi|\le 2$. Define 
\[
E_0 f(x, t):=\int_{\R^n} f(\xi)\chi(\xi) e^{ix\cdot \xi+it|\xi|^{\gamma}} d\xi.
\]
We will prove that
\beq\label{170713a.4}
\|E_0 f\|_{L^p(\R^n\times [-\lambda, \lambda])} \lesim \lambda^{n(\frac{1}{2}-\frac{1}{p})+\epsilon} \|\hat{f}\|_{L^p(\R^n)}
\endeq
for every $\lambda\ge 1$. Once this is proved, a rescaling argument shows that 
\beq\label{170713a.2}
\left(\int_{\R^n\times I} \left| \int_{\R^n}e^{ix\xi}\hat{f}(\xi)e^{it|\xi|^{\gamma}} d\xi\right|^p dxdt\right)^{1/p} \lesim 2^{k\gamma n(\frac{1}{2}-\frac{1}{p})-\frac{k\gamma }{p}+k\epsilon}\|f\|_{L^{p}(\R^n)}
\endeq
whenever $\hat{f}(\xi)$ is supported on the $k$-th annulus, that is $|\xi|\approx 2^k$. As a consequence we obtain Theorem \ref{180320thm1.6}.\\

We will prove \eqref{170713a.4} for every elliptic phase. Let $c_0$ be a small positive real number. 
Let $\phi: \R^n\to \R$ be a smooth function with 
\beq\label{elliptic}
|\phi(\xi)|+|\nabla \phi(\xi)|\lesim 1, c_0 I_n \le (\nabla^2 \phi)(\xi)\le \frac{1}{c_0} I_n \quad \text{ for every } |\xi|<10,
\endeq 
where $I_n$ is the identity matrix of order $n\times n$. Let $\chi_0: \R^n\to \R$ be a non-negative smooth bump function supported on $|\xi|\le 2$. Define 
\[
E_{\phi} f(x, t):=\int_{\R^n} f(\xi)\chi_0(\xi) e^{ix\cdot \xi+it\phi(\xi)} d\xi.
\]
We will prove that
\beq\label{170713a.4-a}
\sup_{\phi: \eqref{elliptic}}\|E_{\phi} f\|_{L^p(\R^n\times [-\lambda, \lambda])} \lesim \lambda^{n(\frac{1}{2}-\frac{1}{p})+\epsilon} \|\hat{f}\|_{L^p(\R^n)}
\endeq
for every $\lambda\ge 1$. 

For a ball $B_{\lambda}\subset \R^{n+1}$ of radius $\lambda$, we will let $B_{\lambda}^-$ denote its projection in the first $n$ variables, i.e. spatial variables. That is, $B_{\lambda}^-$ is a ball of radius $\lambda$ in $\R^n$. We also define the associated weight 
\[
w_{B_{\lambda}^-}(x):=\frac{1}{(1+\frac{\|x-c\|}{\lambda})^{100n}} \text{ for } x\in \R^n.
\]
Here $c$ denotes the center of $B_{\lambda}^-$. We remark that in the argument below, various implicit constants depend on this choice of weight. However, this dependence is not important, and to avoid unnecessary technicalities, we will not make these details explicit. We refer the interested reader to Li \cite{Li17}, which contains all the necessary details that are required to run the Bourgain--Guth argument \cite{MR2860188}. \\

We prove \eqref{170713a.4-a} by an inductive argument. Denote by $Q_{\lambda}$ the smallest constant such that 
\[
\sup_{\phi: \eqref{elliptic}}\|E_{\phi} f\|_{L^p(B_{\lambda})} \le Q_{\lambda} \cdot \lambda^{n(\frac{1}{2}-\frac{1}{p})}\|\hat{f}\|_{L^p(w_{B^-_{\lambda}})}
\]
for every ball $B_{\lambda}\subset \R^{n+1}$ with $B_{\lambda}\subset B^-_{\lambda}\times [-\lambda, \lambda]$. Of course our goal is to prove that 
\beq \label{eq:Qind}
Q_{\lambda}\lesim_{\epsilon} \lambda^{\epsilon}
\endeq
for every $\epsilon>0$. In the following, for the sake of simplicity, we will always abbreviate $E_{\phi} f$ to $E f$.\\

First, we normalize $f$ such that 
\beq\label{normalisation}
\|\hat{f}\|_{L^p(w_{B^-_{\lambda}})} \lambda^{n(\frac{1}{2}-\frac{1}{p})} = 1.
\endeq
Next, let $K_n$ be a large integer that is to be determined, satisfying 
\[
K_n\ll \lambda^{\epsilon}.
\]
For a large dyadic integer $K$, let $\mathrm{Col}_K$ denote the collection of all dyadic cubes of length $1/K$. We write 
\[
Ef=\sum_{\alpha_n\in \mathrm{Col}_{K_n}} Ef_{\alpha_n} \quad \text{with} \quad f_{\alpha_n}:=f\cdot \mathbbm{1}_{\alpha_n}.
\]
Here $\{\mathbbm{1}_{\alpha_n}\}_{n}$ forms a smooth partition of unity, and $\mathbbm{1}_{\alpha_n}$ is supported on $2\alpha_n$. On every ball $B_{K_n}\subset \R^{n+1}$ of radius $K_n$, by the uncertainty principle, we know that $|E f_{\alpha_n}|$ is essentially a constant, for every $\alpha_n\in \mathrm{Col}_{K_n}$. We let $|Ef_{\alpha_n}|(B_{K_n})$ denote this constant. Denote by $\alpha_n^*$ the cube that maximizes 
\[
\{|E f_{\alpha_n}|(B_{K_n})\}_{\alpha_n\in Col_{K_n}}.
\] 
Consider the collection 
\[
\mathrm{Col}^*_{K_n}:=\{\alpha_n\in \mathrm{Col}_{K_n}: |Ef_{\alpha_n}|(B_{K_n})\ge K_n^{-n} |Ef_{\alpha^*_n}|(B_{K_n})\}.
\]
Here the choice of the coefficient $K_n^{-n}$ is not strict. One can also use $K_n^{-2n}$ or something even smaller. \\

\noindent There are three cases. 
\[
\begin{split}
\text{Case 1: } & \text{There exists an integer $1\le j\le n$, and cubes $\alpha_n^{(1)}, \dots, \alpha_n^{(j)}\in \mathrm{Col}^*_{K_n}$ }\\
& \text{which are $\tfrac{1}{K_{n-1}}$-separated such that every cube within $\mathrm{Col}^*_{K_n}$ is }\\
& \text{in the $\tfrac{1}{K_{n-1}}$ neighborhood of some $\alpha^{(j')}_{n}$.}
\end{split}
\]
Here $K_{n-1}\ll K^{\epsilon}_n$  is also to be determined. Next, we have 
\[
\begin{split}
\text{Case 2: }& \text{There exist cubes $\alpha_n^{(1)}, \dots, \alpha_n^{(n+1)}\in \mathrm{Col}^*_{K_n}$ that are $\tfrac{1}{K_{n-1}}$ separated and do}\\
& \text{not lie in the $\tfrac{100}{K_n}$ neighborhood of any $(n-1)$-dimensional subspace.}
\end{split}
\]
If Case 1 and Case 2 are not satisfied, then we have 
\[
\begin{split}
\text{Case 3: } & \text{All cubes in $\mathrm{Col}^*_{K_n}$ lie in the $\tfrac{C(K_{n-1})}{K_n}$ neighborhood of a subspace}\\
&\text{ of dimension $(n-1)$.}
\end{split}
\]
Here $C(K_{n-1})$ is a large constant depending on $K_{n-1}$ which may change from line to line (it always suffices to take, say, $C(K_{n-1})=K_{n-1}^{100n}$).\\

We deal with these three cases separately. In Case 1, we have 
\beq\label{0113e1.17}
|Ef|\lesim \max_{\alpha_n\in Col_{K_n}} |Ef_{\alpha_n}|+\max_{\alpha_{n-1}\in Col_{K_{n-1}}} |Ef_{\alpha_{n-1}}|.
\endeq
In Case 2, we use
\[
|Ef|\lesim K_n^{2n} \left( \prod_{j=1}^{n+1} |Ef_{\alpha_n^{(j)}}|\right)^{\frac{1}{n+1}}.
\]
In Case 3, we use 
\beq\label{0113e1.19}
\begin{split}
& |Ef|\lesim  \max_{\alpha_n\in Col_{K_n}} |Ef_{\alpha_n}| + \max_{\substack{L_{n-1}: \text{ subspace}\\ \text{ of dimension } (n-1)}} \Big|\sum_{\alpha_{n-1}: \text{dist}(\alpha_{n-1}, L_{n-1})\le \frac{1}{K_{n-1}}}Ef_{\alpha_{n-1}}\Big|.
\end{split}
\endeq
In the last summation, we implicitly assumed that $\alpha_{n-1}\in \mathrm{Col}_{K_{n-1}}$, and that 
\[
\frac{C(K_{n-1})}{K_n}\le \frac{1}{K_{n-1}}.
\] Here we agree upon a convention: for $1\le j\le n$, whenever the symbol $\alpha_j$ appears, we always assume that $\alpha_j\in \mathrm{Col}_{K_j}$, to keep notation simpler. Combining \eqref{0113e1.17}--\eqref{0113e1.19}, we obtain
\[
\begin{split}
|Ef|\lesim & \max_{\alpha_n\in \mathrm{Col}_{K_n}} |Ef_{\alpha_n}|+\max_{\alpha_{n-1}\in \mathrm{Col}_{K_{n-1}}} |Ef_{\alpha_{n-1}}| + K_n^{2n} \left( \prod_{j=1}^{n+1} |Ef_{\alpha_n^{(j)}}|\right)^{\frac{1}{n+1}}\\
& + \max_{\substack{L_{n-1}: \text{ subspace}\\ \text{ of dimension } (n-1)}} \Big|\sum_{\alpha_{n-1}: \text{dist}(\alpha_{n-1}, L_{n-1})\le \frac{1}{K_{n-1}}}Ef_{\alpha_{n-1}}\Big|.
\end{split}
\]
We raise both sides to the $p$-th power, then integrate over $B_{K_n}$, and in the end 
sum over balls $B_{K_n}$ inside $B_{\lambda}$,
\beq\label{different-case}
\begin{split}
\int_{B_{\lambda}}|Ef|^p\lesim & \sum_{\alpha_n\in \mathrm{Col}_{K_n}} \int_{B_{\lambda}}|Ef_{\alpha_n}|^p+\sum_{\alpha_{n-1}\in \mathrm{Col}_{K_{n-1}}} \int_{B_{\lambda}}|Ef_{\alpha_{n-1}}|^p\\
& + \sum_{\alpha_n^{(1)}, \dots, \alpha_n^{(n+1)} \text{ in Case 2} }\int_{B_{\lambda}} K_n^{2pn} \left( \prod_{j=1}^{n+1} |Ef_{\alpha_n^{(j)}}|\right)^{\frac{p}{n+1}}\\
& + \sum_{B_{K_n}\subset B_{\lambda}}\max_{L_{n-1}} \int_{B_{K_n}}\Big|\sum_{\alpha_{n-1}: \text{dist}(\alpha_{n-1}, L_{n-1})\le \frac{1}{K_{n-1}}}Ef_{\alpha_{n-1}}\Big|^p.
\end{split}
\endeq
There are four terms on the right hand side. It is the contribution from the last term that gives us the ultimate constraint for the exponent $p$, as stated in \eqref{p-exponent}. \\

Let us be more precise. The first and second summands on the right hand side of \eqref{different-case} can be taken care of by parabolic rescaling. We will deal with the third summand using multilinear restriction estimates due to Bennett, Carbery and Tao \cite{BCT06}. The last term requires further careful analysis.

For the first and second summands, we will apply rescaling. The argument is the same in both cases. Hence we will only write down the rescaling argument for the first summand. 
\[
\int_{B_{\lambda}} |Ef_{\alpha_n}|^p =\int_{B_{\lambda}} \left|\int f_{\alpha_n}(\xi)e^{i\xi x+it\phi(\xi)}d\xi \right|^p dxdt.
\]
Here we apply the change of variable 
\[
\xi\to \frac{\xi}{K_n} + c_{\alpha_n}, \text{ with } c_{\alpha_n} \text{ being the center of } \alpha_n. 
\]
We obtain 
\[
\begin{split}
& K_n^{-np}\int_{B_{\lambda}} \left|\int f_{\alpha_n}(\tfrac{\xi}{K_n}+c_{\alpha_n})e^{i(\tfrac{\xi}{K_n}+c_{\alpha_n}) x+i t \phi(\tfrac{\xi}{K_n}+c_{\alpha_n})}d\xi \right|^p dxdt\\
= & K_n^{-np}\int_{B_{\lambda}} \left|\int f_{\alpha_n}(\tfrac{\xi}{K_n}+c_{\alpha_n})e^{i\tfrac{\xi}{K_n} x+iK_n^2 \phi(\tfrac{\xi}{K_n}+c_{\alpha_n})\frac{t}{K_n^2}}d\xi\right|^p dxdt
\end{split}
\]
Next we apply the change of variables 
\[
x/K_n\to x \text{ and } t/K_n^2\to t
\]
to obtain 
\beq\label{180320e1.30}
K_n^{-np+n+2}\int_{\widetilde{B_{\lambda}}} \left|\int f_{\alpha_n}(\tfrac{\xi}{K_n}+c_{\alpha_n})e^{i\xi x+it\cdot K_n^2  \phi(\frac{\xi}{K_n}+c_{\alpha_n})}d\xi \right|^p dxdt
\endeq
Here $\widetilde{B_{\lambda}}\subset \R^{n+1}$ is a rectangular box of dimensions $\frac{\lambda}{K_n}\times \dots \times \frac{\lambda}{K_n}\times \frac{\lambda}{K_n^{2}}$. The reason of writing it in this form is that 
\[
\widetilde{\phi}(\xi):=K_n^2 \phi(\tfrac{\xi}{K_n}+c_{\alpha_n})-K_n \langle (\nabla \phi)(c_{\alpha_n}), \xi\rangle-K_n^2 \phi(c_{\alpha_n}) \text{ still satisfies } \eqref{elliptic}. 
\]
By a change of variable, \eqref{180320e1.30} can be bounded by 
\[
K_n^{-np+n+2}\int_{2\widetilde{B_{\lambda}}} \left|\int f_{\alpha_n}(\tfrac{\xi}{K_n}+c_{\alpha_n})e^{i\xi x+it \widetilde{\phi}(\xi)}d\xi \right|^p dxdt.
\]

Next we split the rectangular box $2\widetilde{B_{\lambda}}$ into a union of cubes of side-length $\frac{\lambda}{K_n^{2}}$.
\[
\begin{split}
& K_n^{-np}K_n^{n+2}\sum_{B_{\lambda/K_n^{2}}\subset 2\widetilde{B_{\lambda}}}\int_{B_{\lambda/K_n^{2}}} \left|\int f_{\alpha_n}(\tfrac{\xi}{K_n}+c_{\alpha_n})e^{i\xi x+it \widetilde{\phi}(\xi)}d\xi\right|^p dxdt\\
 \lesim & K_n^{-np}K_n^{n+2}  \sum_{B^-_{\lambda/K_n^{2}}} \left(\frac{\lambda}{K_n^{2}}\right)^{np(\tfrac{1}{2}-\frac{1}{p})} Q_{\tfrac{\lambda}{K_n^{2}}}^p \int_{\R^n} \left|\widehat{f_{\alpha_n}(\tfrac{\cdot}{K_n}+c_{\alpha_n})} (x)\right|^p w_{B^-_{\lambda/K_n^2}}(x)dx
\end{split}
\]
Here we applied the induction hypothesis. Notice that in the last summand, we were (essentially) summing over spatial balls of radius $\lambda/K_n^{2}$ on $\R^n$.
\[
\begin{split}
\lesim & K_n^{-np}K_n^{n+2} \left(\frac{\lambda}{K_n^{2}}\right)^{np(\frac{1}{2}-\frac{1}{p})} Q_{\tfrac{\lambda}{K_n^{2}}}^p \int_{\R^n} \left|\widehat{f_{\alpha_n}(\tfrac{\cdot}{K_n}+c_{\alpha_n})} (x)\right|^p w_{B^-_{\lambda/K_n}}(x) dx\\
\lesim & K_n^{-np}K_n^{n+2} K_n^{np-n} \left(\frac{\lambda}{K_n^{2}}\right)^{np(\frac{1}{2}-\frac{1}{p})} Q_{\frac{\lambda}{K_n^{2}}}^p \int_{\R^n} \left|\widehat{f_{\alpha_n}} (\xi)\right|^p w_{B^-_{\lambda}}(x) dx
\end{split}
\]
Sum over $\alpha_n$, we obtain 
\[
\begin{split}
& K_n^{-np} K_n^{n+2} \left(\frac{\lambda}{K_n^{2}}\right)^{np(\frac{1}{2}-\frac{1}{p})} Q^p_{\frac{\lambda}{K_n^{2}}} K_n^{-n} K_n^{np} \sum_{\alpha_n} \int_{\R^n} |\widehat{f_{\alpha_n}}|^p w_{B^-_{\lambda}}\\
\lesim & K_n^{-np} K_n^{n+2} \left(\frac{1}{K_n^{2}}\right)^{np(\frac{1}{2}-\frac{1}{p})} Q^p_{\frac{\lambda}{K_n^{2}}} K_n^{-n} K_n^{np} 
\end{split}
\]
In the last step, we applied the normalization condition \eqref{normalisation}. Let the exponent of $K_n$ be equal to zero and we obtain 
\[
2-2np(\tfrac{1}{2}-\tfrac{1}{p})=0\Rightarrow p= \tfrac{2(n+1)}{n}.
\]
This is exactly the exponent in the Fourier restriction conjecture. Moreover, the last display tells us that, for the contribution from the first and second terms in \eqref{different-case}, the induction can be closed whenever $p> \frac{2(n+1)}{n}$. \\

Now we deal with the third summand on the right hand side of \eqref{different-case}. When $p\ge  \frac{2(n+1)}{n}$, by multilinear restriction of Bennett, Carbery and Tao \cite{BCT06} and by Bernstein's inequality and H\"older's inequality, 
\[
\begin{split}
& \sum_{\alpha_n^{(1)}, \dots, \alpha_n^{(n+1)} \text{ in Case 2} }\int_{B_{\lambda}} K_n^{2n} \left( \prod_{j=1}^{n+1} |Ef_{\alpha_n^{(j)}}|\right)^{\frac{p}{n+1}} \lesim K_{n}^{2n} K_n^{100 n!} \lambda^{\epsilon}.
\end{split}
\]
Again we see that there is no problem for this term as $K_n$ can be chosen to be much smaller compared with $\lambda^{\epsilon}$. \\

In the end, we come to the last summand on the right hand side of \eqref{different-case}. Fix a ball $B_{K_n}\subset \R^{n+1}$. Assume that the maximum is attained at the $(n-1)$ dimensional subspace $L_{n-1}$. We need to consider
\[
\int_{B_{K_n}}\Big|\sum_{\alpha_{n-1}: \text{dist}(\alpha_{n-1}, L_{n-1})\le \frac{1}{K_{n-1}}}Ef_{\alpha_{n-1}}\Big|^p.
\]
Notice that each $|Ef_{\alpha_{n-1}}|$ is essentially a constant on $B_{K_{n-1}}$, a ball of radius $K_{n-1}$ which is much smaller compared with $K_n^{\epsilon}$. Hence tentatively we fix a ball $B_{K_{n-1}}\subset B_{K_n}$. Let $\alpha_{n-1}^*$ denote the cube that maximizes 
\[
\{|E f_{\alpha_{n-1}}|(B_{K_{n-1}})\}_{\alpha_{n-1}: \text{dist}(\alpha_{n-1}, L_{n-1})\le \frac{1}{K_{n-1}}}.
\] 
Consider the collection 
\[
\mathrm{Col}^*_{K_{n-1}}:=\{\alpha_{n-1}: \text{dist}(\alpha_{n-1}, L_{n-1})\le \tfrac{1}{K_{n-1}} \text{ and } |Ef_{\alpha_{n-1}}|(B_{K_{n-1}})\ge K_{n-1}^{-n} |Ef_{\alpha^*_{n-1}}|(B_{K_{n-1}})\}.
\]
\noindent There are three further cases. 
\[
\begin{split}
\text{Case 3.1: } & \text{There exists an integer $1\le j\le n-1$, and cubes $\alpha_{n-1}^{(1)}, \dots, \alpha_{n-1}^{(j)}\in \mathrm{Col}^*_{K_{n-1}}$ }\\
& \text{which are $\tfrac{1}{K_{n-2}}$-separated such that every cube within $\mathrm{Col}^*_{K_{n-1}}$ is }\\
& \text{in the $\tfrac{1}{K_{n-2}}$ neighborhood of some $\alpha^{(j')}_{n-1}$.}
\end{split}
\]
Here $K_{n-2}\ll K^{\epsilon}_{n-1}$  is also to be determined. Moreover, we have 
\[
\begin{split}
\text{Case 3.2: } & \text{There exist cubes $\alpha_{n-1}^{(1)}, \dots, \alpha_{n-1}^{(n)}\in \mathrm{Col}^*_{K_{n-1}}$ that are $\tfrac{1}{K_{n-2}}$ separated}\\
& \text{and do not lie in the $\tfrac{100}{K_{n-1}}$ neighborhood of any $(n-2)$-dimensional subspace.}
\end{split}
\]
If the above two cases are not satisfied, then we must have 
\[
\begin{split}
\text{Case 3.3: } & \text{All cubes in $\mathrm{Col}^*_{K_{n-1}}$ lie in the $\tfrac{C(K_{n-2})}{K_{n-1}}$ neighborhood of a linear subspace}\\
&\text{of dimension $(n-2)$.}
\end{split}
\]
Here $C(K_{n-2})$ is a large constant depending on $K_{n-2}$. It suffices to take $C(K_{n-2})=K_{n-2}^{100n}$.\\

Similarly to \eqref{0113e1.17}--\eqref{0113e1.19}, we have that for every point $(x, t)\in B_{K_{n-1}}$,
\[
\begin{split}
|Ef|\lesim & \max_{\alpha_n\in Col_{K_n}} |Ef_{\alpha_n}|+\max_{\alpha_{n-1}\in Col_{K_{n-1}}} |Ef_{\alpha_{n-1}}| +\max_{\alpha_{n-2}\in Col_{K_{n-2}}} |Ef_{\alpha_{n-2}}|\\
& + K_{n-1}^{2n} \left( \prod_{j=1}^{n} |Ef_{\alpha_{n-1}^{(j)}}|\right)^{\frac{1}{n}} + \max_{L_{n-2}} \Big|\sum_{\alpha_{n-2}: \text{dist}(\alpha_{n-2}, L_{n-2})\le \frac{1}{K_{n-2}}}Ef_{\alpha_{n-2}}\Big|.
\end{split}
\]
We first raise both sides to the $p$-th power, integrate over $B_{K_{n-1}}$, 
and then sum over all balls $B_{K_{n-1}}\subset B_{K_n}$, and in the end sum over all balls $B_{K_n}\subset B_{\lambda}$, 
\[
\begin{split}
\int_{B_{\lambda}}|Ef|^p\lesim & \sum_{\alpha_n\in Col_{K_n}} \int_{B_{\lambda}}|Ef_{\alpha_n}|^p+\sum_{\alpha_{n-1}\in Col_{K_{n-1}}} \int_{B_{\lambda}}|Ef_{\alpha_{n-1}}|^p\\
& +\sum_{\alpha_{n-2}\in Col_{K_{n-2}}} \int_{B_{\lambda}}|Ef_{\alpha_{n-2}}|^p\\
& + \sum_{\alpha_{n-1}^{(1)}, \dots, \alpha_{n-1}^{(n)} \text{ in Case 3.2} }\int_{B_{\lambda}} K_{n-1}^{2n} \left( \prod_{j=1}^{n} |Ef_{\alpha_{n-1}^{(j)}}|\right)^{\frac{p}{n}}\\
& + \sum_{B_{K_n}\subset B_{\lambda}}  \sum_{B_{K_{n-1}}\subset B_{K_n}}\max_{L_{n-2}} \int_{B_{K_{n-1}}}\Big|\sum_{\alpha_{n-2}: \text{dist}(\alpha_{n-2}, L_{n-2})\le \frac{1}{K_{n-2}}}Ef_{\alpha_{n-2}}\Big|^p.
\end{split}
\]
There are five terms on the right hand side of the last display. By the same scaling argument as above, we can handle the first three summands. For the fourth summand, we again apply multilinear restrictions due to Bennett, Carbery and Tao \cite{BCT06}. However, notice that we are applying an $n$-linear restriction estimate in $\R^{n+1}$. This will not give us the restriction exponent $\frac{2(n+1)}{n}$, but something larger. To be precise, we have 
\[
\sum_{\alpha_{n-1}^{(1)}, \dots, \alpha_{n-1}^{(n)} \text{ in Case 3.2} }\int_{B_{\lambda}} K_{n-1}^{2n} \left( \prod_{j=1}^{n} |Ef_{\alpha_{n-1}^{(j)}}|\right)^{\frac{p}{n}} \lesim K_{n-1}^{2n} K_{n-1}^{100 n!} \lambda^{\epsilon},
\]
for every 
\[
p\ge \tfrac{2n}{n-1},
\]
which we assume. \\

Hence it remains to handle the last summand 
\[
\sum_{B_{K_{n-1}}\subset B_{\lambda}}\max_{L_{n-2}} \int_{B_{K_{n-1}}}\Big|\sum_{\alpha_{n-2}: \text{dist}(\alpha_{n-2}, L_{n-2})\le \frac{1}{K_{n-2}}}Ef_{\alpha_{n-2}}\Big|^p.
\]
We repeat this iteration until we reach
\beq\label{apply-decoupling}
\sum_{B_{K_{n-k}}\subset B_{\lambda}}\max_{L_{n-k-1}} \int_{B_{K_{n-k}}}\Big|\sum_{\alpha_{n-k-1}: \text{dist}(\alpha_{n-k-1}, L_{n-k-1})\le \frac{1}{K_{n-k-1}}}Ef_{\alpha_{n-k-1}}\Big|^p,
\endeq
where $k$ is the largest positive integer such that
\beq \label{eq:kchoice}
k \leq \tfrac{n-1}{3};
\endeq
in other words, 
\beq \label{eq:kchoice2}
k = 
\begin{cases}
\frac{n-3}{3} \quad \text{if $n \equiv 0 \pmod 3$} \\
\frac{n-1}{3} \quad \text{if $n \equiv 1 \pmod 3$} \\
\frac{n-2}{3} \quad \text{if $n \equiv 2 \pmod 3$}.
\end{cases}
\endeq
Collecting all the constraints on the exponent $p$ from applying the multilinear restriction estimate, we obtain 
\beq\label{constraint-1}
p\ge \tfrac{2(n-k+1)}{n-k}.
\endeq
Instead of running the previous argument again on \eqref{apply-decoupling}, we apply the decoupling inequalities of Bourgain and Demeter \cite{MR3374964}:
\[
\begin{split}
& \int_{B_{K_{n-k}}}\Big|\sum_{\alpha_{n-k-1}: \text{dist}(\alpha_{n-k-1}, L_{n-k-1})\le \frac{1}{K_{n-k-1}}}Ef_{\alpha_{n-k-1}}\Big|^p\\
\lesim & (K_{n-k-1})^{(n-k-1)(\frac{1}{2}-\frac{1}{p})p + \varepsilon} \sum_{\alpha_{n-k-1}} \int_{B_{K_{n-k}}} |Ef_{\alpha_{n-k-1}}|^p.
\end{split}
\]
The above inequality will hold as long as 
\beq \label{eq:pchoice1}
p \le \tfrac{2(n-k+1)}{n-k-1}.
\endeq 
In the end, we sum over all balls $B_{K_{n-k}}$ inside $B_{\lambda}$, and obtain 
\[ 
(K_{n-k-1})^{(n-k-1)(\frac{1}{2}-\frac{1}{p})p + \varepsilon} \sum_{\alpha_{n-k-1}} \int_{B_{\lambda}} |Ef_{\alpha_{n-k-1}}|^p.
\]
It is clear now that we should apply parabolic rescaling. This gives us
\[
(K_{n-k-1})^{(n-k-1)(\frac{1}{2}-\frac{1}{p})p + \varepsilon} \times (K_{n-k-1})^{2-2np(\frac{1}{2}-\frac{1}{p})} \times Q_{\frac{\lambda}{K_{n-k-1}^2}}^p.
\]
By equating the exponent of $K_{n-k-1}$ with zero we obtain the constraint
\beq \label{eq:pchoice2}
p> \tfrac{2(n+k+3)}{n+k+1}.
\endeq
This constraint is more restrictive than \eqref{constraint-1}, by condition~(\ref{eq:kchoice}). 
By choosing $k$ as in (\ref{eq:kchoice2}), and substituting that into (\ref{eq:pchoice2}), we obtain the constraint on $p$ in (\ref{p-exponent}). To summarize, we have shown that for $p$ satisfying (\ref{p-exponent}) and (\ref{eq:pchoice1}), we can close the induction. Thus, we have established (\ref{eq:Qind}) and therefore (\ref{eq:ls}) for such $p$'s. 
Finally, by the result of Rogers and Seeger \cite{MR2629687}, we already know that (\ref{eq:ls}) holds for all $p>2+\frac{4}{n+1}$ (even with $\varepsilon=0$), which completes the proof of Theorem~\ref{180320thm1.6} for the claimed range in $p$.

\section{Appendix: A maximal multi-frequency estimate of Krause and Lacey} \label{sect:KL}

Fix $\ell_0 \in \Z$ and $d \in \N$ with $d \geq 2$. Decompose $1 = \sum_{\ell \in \Z} \varphi_{\ell}(t)$ where $\varphi_{\ell}(t) := \varphi_0(2^{-\ell} t)$ for some smooth even function $\varphi_0$ supported on $|t| \simeq 1$.
For $u > 0$, let 
$$
T^{(u)} f(x) = \sum_{\ell \geq \ell_0} T^{(u)}_{\ell} f(x)
$$
where
$$
T^{(u)}_{\ell} f(x) := \int_{\R} f(x-t) \varphi_{\ell}(t) e^{i u t^d} \frac{dt}{t}.
$$
Following the method in Bourgain \cite{MR1019960}, we will prove the following maximal multi-frequency estimate.

\begin{thm} \label{thm:KL}
There exists a constant $C$, such that for any $\tau > 0$, $M \in \N$ and any $\theta_1, \dots, \theta_M \in \R$ with $\min_{1 \leq i < j \leq M} |\theta_i - \theta_j| > 2\tau$, we have 
\beq \label{eq:multifreq}
\left \| \sup_{u > 0} \left | \sum_{m=1}^M \Mod_{\theta_m} T^{(u)} P_{\tau} f_m \right | \right \|_{L^2(\R)} \leq C (\log M)^2 \left ( \sum_{m=1}^M \|f_m\|_{L^2(\R)}^2 \right )^{1/2}
\endeq
for all $f_1, \dots, f_M \in L^2(\R)$, where $P_{\tau}$ is the Littlewood-Paley projection onto the frequency interval $[-\tau,\tau]$.
\end{thm}

Here $\Mod_{\theta} f(x)$ is the modulation $\Mod_{\theta} f(x) := e^{i \theta x} f(x)$.

As a corollary, we obtain the Theorem 3.5 of Krause and Lacey \cite{MR3658135}:
\begin{cor}[Krause--Lacey \cite{MR3658135}]
There exists a constant $C$, such that for any $\tau > 0$, $M \in \N$ and any $\theta_1, \dots, \theta_M \in \R$ with $\min_{1 \leq i < j \leq M} |\theta_i - \theta_j| > 2\tau$, we have
$$
\left \| \sup_{u > 0} \left | \sum_{m=1}^M \Mod_{\theta_m} T^{(u)} (P_{\tau}\Mod_{-\theta_m} f) \right | \right \|_{L^2(\R)} \leq C (\log M)^2 \|f\|_{L^2(\R)}
$$
for any $f \in L^2(\R)$.
\end{cor}

Indeed, one can obtain the corollary by applying Theorem~\ref{thm:KL} to $f_m := P_{\tau} \Mod_{-\theta_m} f$, and noting that then $\sum_{m=1}^M \|f_m\|_{L^2(\R)}^2 \leq \|f\|_{L^2(\R)}^2$.

The corollary is slightly stronger than the Theorem 3.5 of Krause and Lacey \cite{MR3658135} because it allows one to take supremum over all $u > 0$ (not just over $u \in (0, \tau^2)$).

To prove Theorem~\ref{thm:KL}, we use the following variant of our Theorem~\ref{main}.

\begin{thm} \label{thm:mainsharpCp}
There exists a constant $C$, such that for all $p \in (2,3)$, we have
\beq \label{eq:VpLpsharpCp}
\left \| V^p \{T^{(u)} f \colon u > 0\} \right \|_{L^p(\R)} \leq C (p-2)^{-1} \|f\|_{L^p(\R)}.
\endeq
\end{thm}

Stein and Wainger proved that
\beq \label{eq:SWsharpCp}
\left \| \sup \{T^{(u)} f \colon u > 0\} \right \|_{L^q(\R)} \leq C_q \|f\|_{L^q(\R)}
\endeq
for $1 < q < \infty$. By complex interpolation, we then get the following Corollary:

\begin{cor} \label{cor:mainsharpCp}
There exists a constant $C$, such that for all $r \in (2,3)$, we have
\beq \label{eq:VrL2sharpCp}
\left \| V^r \{T^{(u)} f \colon u > 0\} \right \|_{L^2(\R)} \leq C (r-2)^{-1} \|f\|_{L^2(\R)}.
\endeq
\end{cor}

Indeed, for any $r \in (2,\infty)$, one can obtain (\ref{eq:VrL2sharpCp}) by interpolating between (\ref{eq:VpLpsharpCp}) with $p = (r+6)/4$, and (\ref{eq:SWsharpCp}) with $q = 3/2$.

Below we first prove Theorem~\ref{thm:mainsharpCp}, and then use Corollary~\ref{cor:mainsharpCp} to prove Theorem~\ref{thm:KL}.

\begin{proof}[Proof of Theorem~\ref{thm:mainsharpCp}]
By the argument in Section~\ref{section:long-jumps}, we have
$$
\sup_{\lambda > 0} \left \| \lambda \sqrt{ N_{\lambda} \{T^{(2^{kd})} f \colon k \in \Z \} } \right \|_{L^p(\R)} \leq C_p \|f\|_{L^p(\R)}
$$
for $1 < p < \infty$. By the real interpolation argument in Lemma 3.3 of Bourgain \cite{MR1019960} (see also Lemma 2.1 of Jones, Seeger and Wright \cite{MR2434308}), we have
\beq \label{eq:longjumpsharpCp}
\left \| V^p \{T^{(2^{kd})} f \colon k \in \Z \} \right \|_{L^p(\R)} \leq C (p-2)^{-1} \|f\|_{L^p(\R)}
\endeq
for all $2 < p < 3$. Furthermore, by the argument in Section~\ref{section:short-jump-large}, we have 
\beq \label{eq:shortjumpsharpCp}
\left \|  \left\| V^p_j T^{(u)}f \right\|_{\ell^p(j \in \Z)} \right \|_{L^p(\R)} \leq C (p-2)^{-1} \|f\|_{L^p(\R)}
\endeq
for all $2 < p < 3$, where $V^p_j T^{(u)}f(x) := V^p \{ T^{(u)} f(x) \colon u \in [2^{jd},2^{(j+1)d}] \}$. Indeed, the left hand side above is bounded by
$$
\sum_{k, \ell \in \Z} \left\| \|V_j^{p} T^{(u)}_{\ell-j} P_{j+k} f \|_{\ell^p {\{j \colon \ell - j \geq \ell_0\}}} \right\|_{L^p(\R)}, 
$$
and the arguments of Sections \ref{sect4.2} and \ref{sect4.5} show that 
\beq \label{eq:shortjumpsharpCp_1}
\sum_{\substack{k, \ell \in \Z \\ \ell \leq -\frac{k}{2(d+1)}}} \left \| \|V_j^{p} T^{(u)}_{\ell-j} P_{j+k} f \|_{\ell^p_j} \right\|_{L^p(\R)} \leq C_p \|f\|_{L^p(\R)}
\endeq
for $1 < p < \infty$, and
\beq \label{eq:shortjumpsharpCp_2}
\sum_{\substack{k, \ell \in \Z \\ \ell > -\frac{k}{2(d+1)}, \, k > \ell(d-1) + C}} + \sum_{\substack{k, \ell \in \Z \\ \ell > -\frac{k}{2(d+1)}, \, k < \ell(d-1) - C}} \left \| \|V_j^{p} T^{(u)}_{\ell-j} P_{j+k} f \|_{\ell^p_j} \right\|_{L^p(\R)} \leq C_p \|f\|_{L^p(\R)}
\endeq
for $2 \leq p < \infty$, where the constants $C_p$ satisfy $\sup_{2 \leq p \leq 3} C_p < \infty$. Furthermore, the arguments in Section \ref{sect4.4} shows that there exist absolute constants $C$ and $\delta > 0$ such that if $\ell \geq 0$ and $k = \ell(d-1) + O(1)$, then
$$
\left \| \|V_j^{p} T^{(u)}_{\ell-j} P_{j+k} f \|_{\ell^p_j} \right\|_{L^p(\R)} \leq C 2^{-\ell \delta d (p-2)} \|f\|_{L^p(\R)}
$$
for all $2 < p < 3$.
Summing these up, we get 
\beq \label{eq:shortjumpsharpCp_3}
\sum_{\ell \geq 0} \sum_{k = \ell(d-1)+O(1)} \left \| \|V_j^{p} T^{(u)}_{\ell-j} P_{j+k} f \|_{\ell^p_j} \right\|_{L^p(\R)} \leq C (p-2)^{-1} \|f\|_{L^p(\R)}
\endeq
for all $2 < p < 3$. (\ref{eq:shortjumpsharpCp}) then follows from (\ref{eq:shortjumpsharpCp_1}), (\ref{eq:shortjumpsharpCp_2}) and (\ref{eq:shortjumpsharpCp_3}). Since 
$$
 V^p \{T^{(u)} f \colon u > 0 \}  \leq  V^p \{T^{(2^{kd})} f \colon k \in \Z \} + \left\| V^p_j T^{(u)}f \right\|_{\ell^p(j \in \Z)},
$$
we obtain the desired conclusion (\ref{eq:VpLpsharpCp}) from (\ref{eq:longjumpsharpCp}) and (\ref{eq:shortjumpsharpCp}).
\end{proof}

\begin{proof}[Proof of Theorem~\ref{thm:KL}]
We will deduce Theorem~\ref{thm:KL} from Theorem~\ref{thm:mainsharpCp}, following  Bourgain \cite{MR1019960} closely (see Lemma 4.13 there). Suppose $\tau > 0$, and  $\theta_1, \dots, \theta_M \in \R$ are such that $\min_{1 \leq i < j \leq M} |\theta_i - \theta_j| > 2\tau$. First, to prove (\ref{eq:multifreq}), it will suffice to show that
\beq \label{eq:multifreq2}
\begin{split}
& \left( \fint_{w \in [0,\frac{1}{100\tau}]} \left \| \sup_{u > 0} \left | \sum_{m=1}^M e^{i \theta_m x} T^{(u)} P_{\tau} f_m (x + w) \right| \right \|_{L^2(\R)}^2  dw \right)^{1/2}  \\
\leq & \, C (\log M)^2 \left ( \sum_{m=1}^M \|f_m\|_{L^2(\R)}^2 \right )^{1/2}
\end{split}
\endeq
Morally speaking, this is the uncertainty principle at work: note that for every $u > 0$, the function $x \mapsto T^{(u)} P_{\tau} f_m(x)$ has Fourier support contained in $[-\tau,\tau]$, and hence $|\Mod_{\theta_m} T^{(u)} P_{\tau} f_m(x)|$ can be thought of as locally constant on an interval of length $\simeq 1/\tau$. To be precise, for each $w \in [0,\frac{1}{100\tau}]$, we have, by Plancherel's identity, that
$$
\|P_{\tau} f_m(\cdot) - P_{\tau} f_m(\cdot +w)\|_{L^2} \leq \frac{1}{2} \|f_m\|_{L^2}
$$
whenever $w \in [0, \frac{1}{100\tau}]$. Thus if $B$ is the best constant for which (\ref{eq:multifreq}) holds, then for all $w \in [0,\frac{1}{100\tau}]$, we have
\[
\begin{split}
& \left \| \sup_{u > 0} \left | \sum_{m=1}^M \Mod_{\theta_m} T^{(u)} P_{\tau} f_m(x) \right | \right \|_{L^2(\R)} \\
\leq & \left \| \sup_{u > 0} \left | \sum_{m=1}^M e^{i \theta_m x} T^{(u)} P_{\tau} f_m(x+w) \right | \right \|_{L^2(\R)} + \frac{B}{2}  \left ( \sum_{m=1}^M \|f_m\|_{L^2(\R)}^2 \right )^{1/2},
\end{split}
\]
so taking $L^2$ average over all $w \in [0,\frac{1}{100\tau}]$, and using (\ref{eq:multifreq2}), we have
$$
B \leq C (\log M)^2 + B/2,
$$
i.e. $B \leq 2 C (\log M)^2$ as desired. Thus, it remains to establish (\ref{eq:multifreq2}), which can be rewritten as
\beq \label{eq:multifreq3}
\begin{split}
& \left \| \left ( \fint_{w \in [0,\frac{1}{100\tau}]} \sup_{u > 0} \left | \sum_{m=1}^M e^{-i \theta_m w} e^{i \theta_m x} T^{(u)} P_{\tau} f_m (x) \right |^2 dw \right)^{1/2} \right \|_{L^2(\R)}  \\
\leq & \, C (\log M)^2 \left ( \sum_{m=1}^M \|f_m\|_{L^2(\R)}^2 \right )^{1/2}
\end{split}
\endeq
by first changing variable $x \mapsto x - w$, and then interchanging the integrals in $x$ and~$w$.

To prove (\ref{eq:multifreq3}), for each $x \in \R$, consider the (bounded) set $A_x \subset \R^M$, given by
$$
A_x := \{(T^{(u)} P_{\tau} f_1(x), \dots, T^{(u)} P_{\tau} f_M(x)) \colon u > 0\}.
$$ 
If $\lambda$ is bigger than the diameter of $A_x$, let $E_{\lambda}(x) = 0$; otherwise let $E_{\lambda}(x)$ be the minimal number of balls in $\R^M$ of radius $\lambda$ that is required to cover $A_x$. ($E_{\lambda}(x)$ is sometimes called the entropy number.) One then observes that for every $s \in \Z$, there exists a finite subset $B_s(x) \subset A_x - A_x$, of cardinality at most $E_{2^s}(x)$, such that
$$
|b_s| \leq 2^{s+1} \quad \text{for every $b_s \in B_s(x)$},
$$
and such that every element $a$ of $A_x$ admits a decomposition
$$
a = \sum_{s \in \Z} b_s \quad \text{with $b_s \in B_s(x)$ for every $s \in \Z$}.
$$
Then the left hand side of (\ref{eq:multifreq3}) is bounded by
\[
\left \| \sum_{s \in \Z} \max_{b_s \in B_s(x)} \left ( \fint_{w \in [0,\frac{1}{100\tau}]} \left | \sum_{m=1}^M e^{-i\theta_m w} e^{i\theta_m x} b_{s,m} \right |^2 dw \right)^{1/2} \right \|_{L^2(\R)};
\]
here $b_s = (b_{s,1}, \dots, b_{s,m})$. Using Cauchy--Schwarz for the sum over $m$, the above display equation is further bounded by
\beq \label{eq:L2min}
\left \| \sum_{s \in \Z} \min \left \{ 2^{s+1} M^{1/2}, \left ( \sum_{b_s \in B_s(x)} \fint_{w \in [0,\frac{1}{100\tau}]} \left | \sum_{m=1}^M e^{-i\theta_m w} e^{i\theta_m x} b_{s,m} \right |^2 dw \right)^{1/2} \right \} \right \|_{L^2(\R)}.
\endeq
To estimate the integral in $w$ above, observe that from the separation of the $\theta_1, \dots, \theta_M$, we have
$$
\left( \fint_{w \in [0,\frac{1}{100\tau}]} \left| \sum_{m=1}^M e^{-i\theta_m w} c_m \right|^2 dw \right)^{1/2} \lesssim \left( \sum_{m=1}^M |c_m|^2 \right)^{1/2};
$$
indeed, the key is that if $\mathfrak{h} \colon \R^M \to \R^M$ is defined by
$$
(\mathfrak{h}a)_m := \sum_{m=1}^M \frac{\tau}{\theta_m - \theta_n} a_n,
$$
then the operator norm of $\mathfrak{h}$ is bounded independent of $M$, which can be deduced, for instance, by comparing it to the (continuous) Hilbert transform on $\R$. Thus (\ref{eq:L2min}) is bounded by
\beq \label{eq:BourgainLemma3.33}
\left \| \sum_{s \in \Z} \min \left \{ 2^{s+1} M^{1/2}, 2^{s+1} E_{2^s}(x)^{1/2} \right \} \right \|_{L^2(\R)}
= 2 \left \| \sum_{s \in \Z} 2^s \min \left \{ M^{1/2}, E_{2^s}(x)^{1/2} \right \} \right \|_{L^2(\R)}.
\endeq
Now let $$F(x) := \left( \sum_{m=1}^M \sup_{u > 0} |T^{(u)} P_{\tau} f_m(x)|^2 \right)^{1/2}.$$ Then the diameter of $A_x$ is at most $2 F(x)$. Hence the entropy number $E_{2^s}(x) = 0$ whenever $2^s > 2 F(x)$. We are then led to sum
$$
\sum_{2^s \leq 2F(x)} 2^s \min \{ M^{1/2}, E_{2^s}(x)^{1/2} \}.
$$
We split this sum into two, one where $2^s \leq M^{-1/2} F(x)$, and another where $M^{-1/2} F(x) \leq 2^s \leq 2 F(x)$. The former sum is bounded by $F(x)$, while the latter sum is bounded by $(\log M) M^{\frac{1}{2}-\frac{1}{r}} \sup_{s \in \Z} 2^s E_{2^s}(x)^{\frac{1}{r}}$ for any $r \in (2,\infty)$. Now pick $r \in (2,\infty)$ such that $$\frac{1}{2} - \frac{1}{r} = (\log M)^{-1},$$ so that $M^{\frac{1}{2}-\frac{1}{r}} \simeq 1$. Then (\ref{eq:BourgainLemma3.33}) is bounded by
$$
\left \| F(x)  + (\log M) \sup_{s \in \Z} 2^s E_{2^s}(x)^{1/r} \right \|_{L^2(\R)}.
$$
But by Stein and Wainger's inequality (\ref{eq:SWsharpCp}), we have
$$
\|F \|_{L^2(\R)} \lesssim  \left ( \sum_{m=1}^M \|f_m\|_{L^2(\R)}^2 \right )^{1/2}.
$$
Furthermore, one can relate the entropy $E_{2^s}(x)^{1/r}$, with the $r$-th variation norm pointwise:
$$
\sup_{s \in \Z} 2^s E_{2^s}(x)^{1/r} \leq \left( \sum_{m=1}^M |V^r \{T^{(u)} P_{\tau} f_m(x) \colon u > 0\} |^2 \right)^{1/2}.
$$
Hence
$$
\left \|(\log M) \sup_{s \in \Z} 2^s E_{2^s}(x)^{1/r} \right \|_{L^2(\R)}
\leq (\log M) \left( \sum_{m=1}^M \|V^r \{T^{(u)} P_{\tau} f_m \colon u > 0\} \|_{L^2(\R)}^2 \right)^{1/2}.
$$
By Corollary~\ref{cor:mainsharpCp}, the latter is bounded by
$$
C (\log M) (r-2)^{-1} \left ( \sum_{m=1}^M \|f_m\|_{L^2(\R)}^2 \right )^{1/2},
$$
and since $(r-2)^{-1} \simeq \log M$ by our choice of $r$, this completes the proof of Theorem~\ref{thm:KL}.
\end{proof}

\newcommand{\etalchar}[1]{$^{#1}$}

\end{document}